\documentclass[10pt]{amsart}%
\usepackage{amsfonts}
\usepackage{amsmath}
\usepackage{amssymb}
\usepackage{graphicx}
\usepackage{enumerate}[0,T]
\usepackage{multicol}
\usepackage{color}
\usepackage{xcolor}

\usepackage{mathtools}
\usepackage{pstricks, pst-node, pst-text, pst-3d}
\usepackage[arrow, matrix, curve, xdvi, dvips]{xy}

\definecolor{ufogreen}{rgb}{0.24, 0.80, 0.44}
\definecolor{uablue}{rgb}{0.0, 0.2, 0.67}
\definecolor{purpleheart}{rgb}{0.41, 0.21, 0.61}
\definecolor{pinegreen}{rgb}{0.0, 0.47, 0.44}

\newcommand{\comp}{{\mathbb{C}}}

\newcommand{\ds}{\displaystyle}

\newtheorem{corollary}{Corollary}
\newtheorem{lemma}{Lemma}
\newtheorem{definition}{Definition}
\newtheorem{proposition}{Proposition}
\newtheorem{remark}{Remark}

\DeclareMathOperator{\SO}{SO}
\DeclareMathOperator{\GL}{GL}
\DeclareMathOperator{\spn}{span}

\DeclareMathOperator{\pr}{pr}
\DeclareMathOperator{\SE}{SE}

\DeclareMathOperator{\so}{\mathfrak{so}}
\DeclareMathOperator{\tr}{tr}

\newcommand{\be}{\begin{equation}}
\newcommand{\ee}{\end{equation}}
\newcommand{\bean}{\begin{equation*}}
\newcommand{\ean}{\end{equation*}}
\newcommand{\ba}{\begin{array}}
\newcommand{\ea}{\end{array}}
\newcommand{\bpm}{\begin{pmatrix}}
\newcommand{\epm}{\end{pmatrix}}
\newcommand{\Ad}{\text{\rm Ad}}
\newcommand{\St}{\text{\rm St}}

\newcommand{\Id}{\text{\rm Id}}

\newcommand{\SL}{\text{\rm SL}}
\newcommand{\SU}{\text{\rm SU}}

\newcommand{\diag}{\text{\rm diag}}

\newcommand{\fg}{\mathfrak {g}}
\newcommand{\fh}{\mathfrak {h}}

\newcommand{\fp}{\mathfrak {p}}
\newcommand{\fso}{\mathfrak{so}}
\newcommand{\fsu}{\mathfrak{su}}

\newcommand{\fgl}{\mathfrak{gl}}
\newcommand{\fsl}{\mathfrak{sl}}

\newcommand\cD{\mathcal{D}}

\newcommand\cH{\mathcal{H}}
\newcommand\cP{\mathcal{P}}

\newcommand\cM{\mathcal{M}}

\newcommand\IC{{\mathbb C}}

\newcommand\IR{{\mathbb R}}

\newcommand{\M}{\widehat{M}}
\newcommand{\R}{\mathbb{R}}

\begin{document}

\title{Symmetric spaces rolling  on flat spaces}
\author{V. Jurdjevic}
\address{V. Jurdjevic
\newline \hspace*{0,8 cm}Department of Mathematics, University of Toronto, 
 \newline \hspace*{0,8 cm}Toronto, Ontario  M5S 3G3, Canada}
\email{jurdj@math.utoronto.ca}

\author{I. Markina}
\address{I. Markina
\newline \hspace*{0,8 cm}Department of Mathematics, University of Bergen, 
 \newline \hspace*{0,8 cm}P.O.~Box 7803, Bergen N-5020, Norway}
\email{irina.markina@uib.no}


\author{F. Silva Leite}
\address{F. Silva Leite
\newline \hspace*{0,8 cm}Institute of Systems and Robotics, Uni\-versity of Coimbra,
 \newline \hspace*{0,8 cm}P\'olo II, 3030-290 Coimbra, Portugal 
 \newline \hspace*{0,8 cm} and
 \newline \hspace*{0,8 cm}Department of Mathematics, University of Coimbra,
 \newline \hspace*{0,8 cm}Apartado 3008, 3001-501 Coimbra, Portugal}
\email{fleite@mat.uc.pt}
\date {}

\begin{abstract}

The objective of the current paper is essentially two\-fold.  Firstly, to make clear the difference between two notions of rolling a Riemannian manifold over another, using a language accessible to a wider audience, in particular to readers with interest in applications. Secondly, we concentrate on rolling an important class of Riemannian manifolds. In the first part of the paper, the relation between intrinsic and extrinsic rollings is explained in detail, while in the second part we address rollings of symmetric spaces on flat spaces and complement the theoretical results with illustrative examples.
\end{abstract}
\maketitle

\vspace*{0.25 cm}
\noindent \emph{Keywords:} Semi-Riemannian manifolds, symmetric spaces, group actions, Car\-tan decomposition, intrinsic rolling, extrinsic rolling, no-slip, no-twist,  Stiefel manifolds.\vspace*{0.25 cm}

\noindent \emph{2020 Mathematics Subject Classification.} Primary: 53C35, 53A35, 53C50, 53B21, 53C25; Secondary: 53A17, 53C17.


\section{ Introduction}

 In the contemporary literature, there exist two notions of rolling a Riemannian manifold over another, which more recently have also been extended to the semi-Riemannian case.  One of these notions is \emph{intrinsic rolling}, that  doesn't require that the Riemannian manifolds are embedded. This concept uses the intrinsic geometry of the manifolds only, and for  Riemannian surfaces was introduced by Agrachev and Sachkov in \cite{Ag} and by  Bryant and Hsu in \cite{BH}, and later studied for manifolds of higher dimensions, for instance, in~\cite {CMK}, \cite{ChK} and \cite{MOL}. Extension to the semi-Riemannian situation appeared in \cite{MarLeite}.

Another definition of rolling initiated by K. Nomizu in~\cite{N1} and presented more formally by R.W. Sharpe in~\cite{Sh} is the \emph{extrinsic rolling}, which makes use of the isometric embedding of the manifolds in an ambient (semi)-Euclidean vector space $V$, so that the rolling is des\-cribed in terms of the action of the group $\SE(V)$ of oriented isometries of $V$. More recent works that use the extrinsic rolling are, for instance,~\cite{Zimm, HL, JZ, HKSL, CL2012, LLouro, KML, ML2022}. As far as we know, only in \cite{MOL} and \cite{MarLeite} both notions of rolling were addressed, the first for the Riemannian case and the second for the semi-Riemannian case. These two works are rather theoretical for researchers interested in applications of rolling motions but that do not have a strong background in differential geometry.

The purpose of the current paper is essentially twofold.  Firstly, we  want to elucidate the difference between the two notions of rolling using a language that is more accessible to those with practical interest in rolling motions but less familiar with semi-Riemannian geometry. Secondly,  we make transparent the relation between the geometry of the symmetric spaces and its rolling on flat spaces. Our main message is that the transitive action $\tau$ of a Lie group on a symmetric space completely defines its rolling along a chosen curve on the manifold. The differential map $d\tau$ is the isometry between the tangent spaces (after an identification of related vector spaces), that also matches the parallel vector fields on the rolling curves. Examples of semi-Riemannian symmetric spaces are provided, together with  how to construct both types of  rollings.
It is always assumed throughout the paper  that nonholonomic constraints of no-slip and no-twist are required in both situations, and 
the rollings are confined to semi-Riemannian manifolds.

The organization of the paper is the following. After  setting the notation, we discuss in Section~\ref{versus} the intrinsic rolling versus extrinsic. Section~\ref{sec:RolSymSp} is dedicated to rolling of symmetric spaces on flat spaces. Finally, we include Section~\ref{RollingStiefel} with the rolling of Stiefel mani\-folds, in order to illustrate the difference in the construction of rolling motions for a reductive homogeneous manifold, that is not a symmetric semi-Riemannian manifold.
 
\section{Background and notations}\label{sec:background}
In this section we revisit the most important known concepts and results that will be used in the paper, and introduce the necessary notations.
The main reference is the book of O'Neill~\cite{ONeill}, where the reader may find further details.


 \subsection{Semi-Riemannian manifolds}
 
A semi-Riemannian manifold $M$ is a smooth manifold endowed with a non-degenerate symmetric tensor $g(.\,,.)$. We write $n=\dim M$, and denote by $p$  the number of positive eigenvalues of the tensor $g$, so that  $n-p$ is the number of negative eigenvalues of $g$. The crucial example of a semi-Riemannian manifold is the semi-Euclidean vector space $\mathbb R^{p,n-p}$ with the semi-Euclidean product 
$$
\langle x,y\rangle_{p,n-p}=\sum_{k=1}^{p}x_ky_k-\sum_{k=p+1}^{n}x_ky_k,\quad x,y\in \mathbb R^{p,n-p}.
$$
Another important example is a vector space $V$ with a bilinear symmetric non-degenerate form $(.\,,.)_{p,n-p}$. We will often  refer to $(.\,,.)_{p,n-p}$ as a scalar product and simply write $(.\,,.)$ in case there is no  need to specify the signature.

Let $(V,(.\,,.)_{p,n-p})$ be a semi-Riemannian vector space. The isometric embedding map will be denoted by
$$
\iota\colon M\to V.
$$
On existence of such an embedding see~\cite{Clarke}. 
For the moment, we will identify the manifold $M$ with its image under the embedding, that is, notationally, $\iota (M)=M$. 
The semi-Riemannian metric $g(.\,,.)$ on the embedded manifold $M$ is inherited from the semi-Riemannian product $(.\,,.)_{p,n-p}$ in the ambient space $V$.

The isometric embedding of $M$ into $V$ splits the tangent space of $V$, at a point $m \in M$, into a direct sum:
\begin{equation} \label{split}
T_mV = T_mM \oplus (T_mM)^\perp, \quad m \in M,
\end{equation}
where $^\perp$ denotes the orthogonal complement with respect to $(.\,,.)_{p,n-p}$. Note that the tangent space $T_mM$ and the normal space $(T_mM)^\perp$ are nondegenerate subspaces of $(V,(.\,,.)_{p,n-p})$. According to this, any vector $v \in T_mV,$ $m \in M$, can be written uniquely as the sum $v = v^\top+v^\bot$, where $v^\top\in T_mM$, $v^\bot\in (T_mM)^\perp$. 

In what follows, $\overline{\nabla}$ denotes the Levi-Civita connection on the ambient space $V$, and $\nabla$  for the Levi-Civita connection on $M$.
If $X$ and $Y$ are tangent vector fields on $M$, and $\Upsilon$ is a normal vector field on $M$, then
$$\nabla_X Y(p) = \left(\overline{\nabla}_{\bar{X}} \bar{Y}(p) \right)^\top,\quad\nabla_X^\bot \Upsilon(p) = \left(\overline{\nabla}_{\bar{X}} \bar{\Upsilon}(p) \right)^\bot, \quad p \in M,$$
where $\bar{X}$, $\bar{Y}$ and $\bar{\Upsilon}$ are any local extensions to $V$ of the vector fields $X$, $Y$ and ${\Upsilon}$, respectively. 
If $Z(t)$ and $\Upsilon(t)$ are vector fields along a curve $\alpha(t)$, we use $\frac{D_{\alpha(t)}}{dt} Z(t)$ to denote the covariant derivative of $Z(t)$ along $\alpha(t)$ and $\frac{{D}_{\alpha(t)}^\bot}{dt} \Upsilon(t)$ for the normal covariant derivative of $\Upsilon(t)$ along $\alpha(t)$ (these notations are according to~\cite[p. 119]{ONeill}). Again, to simplify notations, in cases where it is clear what is the curve along which the covariant derivative is considered, we may simply write $\frac{{D}}{dt}$ and $\frac{{D}^\bot}{dt}$ instead of the above. 
 Observe that an isometric imbedding of $M$ into $V$ induces the equalities
\begin{equation}\label{eq:cov_derivatives}
\frac{D}{dt}\, Z(t) = \left(\frac{d}{dt}\, Z(t) \right)^\top, \qquad\frac{D^\bot}{dt}\, \Upsilon(t) = \left(\frac{d}{dt}\, \Upsilon(t) \right)^\bot.
\end{equation}

A tangent vector field $Y(t)$ along an absolutely continuous curve $\alpha(t)$ is {\it tangent parallel} if $\frac{D}{dt} Z(t) = 0$, for almost every $t$. Similarly, a normal vector field $\Upsilon(t)$ along $\alpha(t)$ is {\it normal parallel} if $\frac{D^\bot}{dt} \Upsilon(t)=0$, for almost every $t$.

From now on we assume that all curves are absolutely continuous on some real interval $I=[0,T]$, $T >0$ and, even if not said, conditions involving derivatives are valid only for values of the parameter $t$ for which they are well defined.

We denote by $\SE(V)$ the Lie group of semi-Riemannian isometries of the space $(V,(.\,,.)_{p,n-p})$. It can be shown that $SE(V)=\SO(V)\ltimes V$, where, by abuse of notation, $V$ is the subgroup of translations on the vector space $V$, and $\SO(V)$ is the connected component containing identity $e$ of the group of isometries $O(V)$, preserving the orientation of both positive definite and negative definite subspaces of $V$. 
Elements in  $\SE(V)$ will be represented by pairs $g =(R,s)$, $R\in \SO(V)$, $s\in V$, and in this representation the action of $\SE(V)$ on $V$ is  denoted by $(g,v)\mapsto g.v:=R.v+s$, $v\in V$, where $(R,v)\mapsto R.v$ denotes the action of $\SO(V)$ on $V$.  The group product in  $\SE(V)$ is defined as $(R_2,s_2)(R_1,s_1)=(R_2R_1, s_2+R_2.s_1)$.  It then follows that $(e,0)$ is the group identity in $\SE(V)$, and $(R,s)^{-1}=(R^{-1},-R^{-1}.s)$.


\section{Intrinsic versus extrinsic rolling}\label{versus}


We want to recall the definition of a rolling of a semi-Riemannian manifold $M$ over a  semi-Riemannian manifold $\M$ along a given curve $\alpha\colon I\to M$ with the restrictions of no-slip and no-twist.
There are two notions of such a rolling, commonly referred to as "intrinsic" and "extrinsic", that currently exist in the literature. Intrinsic rolling of the Riemannian manifold, introduced in~\cite{Ag,BH}, and also used, for ins\-tance, in~\cite{CMK,MOL}. The intrinsic rolling of  semi-Riemannian manifolds was studied in~\cite{MarLeite}, and the extrinsic rolling of particular families of semi-Riemannian manifolds was treated in~\cite{CL2012, JZ, KL2011,  ML2022}.  The difference between the two definitions is that an "intrinsic" rolling doesn't require that semi-Riemannian manifolds $M$ and $\M$ are isometrically embedded into $(V,(.\,,.)_{p,n-p})$, meanwhile an "extrinsic" rolling presumes such an embedding.

In the following definitions, the semi-Riemannian manifolds $(M,g)$ and $(\M,\widehat g)$ have equal dimension and the semi-Riemannian metric tensors have equal signature. We call an isometry $A\colon T_mM\to T_{\widehat m}\M$ oriented if it preserves the orientation of the positive definite and the negative definite subspaces of $T_mM$ and $T_{\widehat m}\M$.
 \begin{definition} \label{def_intrinsic}  {\bf{Intrinsic rolling.}} A curve $\alpha (t)$ on $M$ is said to roll on a curve $\widehat{\alpha}(t)$ on $\M$ if there exists  an oriented  isometry $A(t): T_{\alpha (t)}M\to T_{\widehat{\alpha} (t)}\M$ such that  
 \be\label{iso}
 \dot{\widehat{\alpha}}(t)=A(t)\dot\alpha (t), \quad \text{and}\ee
$$ 
\ba{l}\text{$A(t)X(t)$ is a parallel vector field in $\M$ along $\widehat{\alpha} (t)$  if and}\\ \text{only if $X(t)$ is a parallel vector field in $M$ along $\alpha (t)$}.
  \ea
  $$  
  The triplet  $(\alpha(t),\widehat{\alpha}(t),A(t))$ is called a \it{ rolling curve}. \end{definition}

In the Riemannian case, the definition of ``extrinsic" rolling initiated by K. Nomizu~\cite{N1} and presented more formally by R.W. Sharpe~\cite{Sh}, makes use of the isometric embedding of $M$ and $\M$ in an ambient space.
Here we use a definition of extrinsic rolling that is more general then that used by~\cite{Sh}. This extended class of rollings is better suited for making the bridge with control theory and also for comparison with Definition~\ref{def_intrinsic}. It includes the presence of a semi-Riemannian vector space $(V,(.\,,.)_{p,n-p})$, orientability for the group of semi-Riemannian motions $\SE(V)$, and the replacement of  piecewise continuous curves  by absolutely continuous curves, see also~\cite{MOL, MarLeite}.

\begin{definition}{\bf{Extrinsic Rolling.}} \label{def_extrinsic}
 An absolutely continuous curve   $g(t)$ in $\SE(V)$, defined on an interval $I=[0,T]$,  is said to roll a curve $\alpha(t) $ in $M$  onto a curve $\widehat{\alpha}(t)$ in  $\M$, without slipping and without twisting, if
 \begin{itemize}
\item [ 1] $g(t)\alpha(t)=\widehat{\alpha}(t)$, for all $  t\in I$,
\item[ 2]  $d_{\alpha(t)}g(t)\, T_{\alpha(t)}M = T_{\widehat{\alpha}(t)} \M$, for all $t\in I$.
\item [ 3] No-slip condition:
$$\dot{\widehat{\alpha}}(t)= d_{\alpha(t)}g(t)\, \dot{\alpha}(t),\,\, \text{for almost every $t$};$$
\item [ 4] No-twist condition (tangential part)
$$d_{\alpha(t)}g(t)\, \frac{D}{dt}\, Z(t) = \frac{D}{dt}\, d_{\alpha(t)}g(t)\, Z(t),$$
for any tangent vector field $Z(t)$ along $\alpha(t)$ and almost every $t$;
\item [ 5] No-twist condition (normal part)
$$d_{\alpha(t)}g(t)\,\frac{D^\bot}{dt}\, \Psi(t) = \frac{D^\bot}{dt}\, d_{\alpha(t)}g(t)\, \Psi(t),$$
for any normal vector field $\Psi(t)$ along $\alpha(t)$ and almost every $t$;
\item [ 6]
$d_{\alpha(t)} g(t)|_{T_{\alpha(t)} M}:T_{\alpha(t)} M \to T_{\widehat{\alpha}(t)} \M$ is orientation preserving.
\vspace*{0,3 cm}

\noindent The curve $g(t)$ that satisfies the  above conditions is called {\emph{rolling map}} along the curve $\alpha(t)$ (also called {\emph{rolling curve}}), and $\widehat{\alpha}(t)$ is the {\emph{development}} of $\alpha (t)$ on $\M$.
\end{itemize}
\end{definition}
  From now on, we may refer to rolling without slipping and without twisting simply as ``rolling''. The first two conditions in Definition \ref{def_extrinsic} are called  ``rolling conditions''.  Notice that the second rolling condition and the splitting~\eqref{split} also implies:
\begin{equation}\label{new2}
 d_{\alpha(t)}g(t)(T_{\alpha(t)}M)^\bot = (T_{\widehat{\alpha}(t)} \\\widehat M)^\bot.
\end{equation}
The no-slip and no-twist conditions can be seen as nonholonomic constraints. They give rise to equations for the velocity of the rolling map, usually called the {\emph{kinematic equations of rolling}.

At first glance, the no-slip and no-twist  conditions in Definition~\ref{def_extrinsic} may look different from those in~\cite{Sh}, however they are equivalent, as proven in \cite{MOL, MarLeite}.  When dealing with concrete examples these nonholonomic constraints are easier to handle when written as in~\cite{Sh}. For that reason, after introducing some necessary notations, we rewrite conditions 3, 4 and 5 in Definition \ref{def_extrinsic} using the terminology in~\cite{Sh}.

 For each action $g(t)=(R(t),s(t))\in \SE(V)$ on $V$, defined by $g(t).p=R(t).p+s(t)$,   the differential (or tangent map) of $g(t)$ at $p\in V$ is given by
\be \label{notation-sharpe00}
d_{p}g(t)v:= \left.\frac{d}{d\epsilon}
    g(t).p(\epsilon)\right|_{\epsilon=0}=R(t).v,\ee
where $\epsilon\mapsto p(\epsilon)$ is a curve in $V$ satisfying $p(0)=p, \,\frac{dp}{d\epsilon}(0)=v$.
If $\dot{g}(t)$ denotes the time derivative of the curve $g(t)$, i.e., 
$$
\dot g(t).p=\left.\frac{d}{d\epsilon}g(\epsilon).p\right|_{\epsilon=t}=\dot R(t).p+\dot s(t),
$$
then, since  $g^{-1}=(R^{-1},-R^{-1}.s)$, we can  define 
\be \label{notation-sharpe0}
\left(\dot g(t) g^{-1}(t)\right).p:=\left.\frac{d}{d\epsilon}
    g(\epsilon) .(g^{-1}(t).p) \right|_{\epsilon=t}=\dot R(t)R^{-1}(t).(p-s(t))+\dot s(t),
\ee so that
\begin{equation} \label{notation-sharpe}
d_{p}\left(\dot g(t) g^{-1}(t)\right).v:=\left.\frac{d}{d\epsilon}
    \left(\dot g(t)g^{-1}(t)\right).p(\epsilon)\right|_{\epsilon=0}=\dot R(t)R^{-1}(t). v.
\end{equation} 
\begin{proposition}\label{prop-sharpe}
Conditions \emph{3}, \emph{4} and \emph{5} in Definition \ref{def_extrinsic} are, respectively, equivalent to:
\begin{itemize}
\item[3'] No-slip condition:
 $$(\dot{g}(t)g^{-1}(t)).\widehat\alpha(t)=0,\,\, \text{for almost every $t$};$$
\item[4'] No-twist condition (tangential part):
 $$d_{\widehat\alpha(t)}(\dot {g}(t)g^{-1}(t))\, T_{\widehat\alpha(t)}{\M}\subset (T_{\widehat\alpha(t)}\M )^\perp, \,\, \text{for almost every $t$};$$ 
\item [ 5'] No-twist condition (normal part):
$$d_{\widehat\alpha(t)}(\dot {g}(t)g^{-1}(t))\, (T_{\widehat\alpha(t)}{\M})^\perp\subset T_{\widehat\alpha(t)}{\M}, \,\, \text{for almost every $t$};$$ 
\end{itemize}
\end{proposition}
It was proved in~\cite{Sh} that given a curve $\alpha (t)$ in $M$ there always exists a unique rolling map $g(t)$ that rolls a Riemannian manifold $M$ on a Riemannian manifold $\M$ along $\alpha$.} The proof can be literally extended to the rolling of semi-Riemannian manifolds, since the arguments in~\cite{Sh} do not rely on the  positive definite property of the metric tensor, but rather on being non-degenerate.
 \begin{remark}
The definition of rolling map doesn't exclude the possibility that $g (t)$ is the identity in $\SE(V)$, otherwise the existence of a rolling map for each curve 
in $M$ would not be guaranteed. This is clearly seen, for instance, in the system consisting of a cylinder rolling on the tangent plane at a point, when the rolling curve lies in the straight line of intersection of the two manifolds.
 \end{remark}
The no-twist conditions in Definition \ref{def_extrinsic} can also be rewritten in terms of parallel vector fields as follows. This is particularly important for the comparison with the intrinsic rolling.

\begin{proposition}\label{prop-notwist}
Conditions \emph{4} and \emph{5} of Definition \ref{def_extrinsic} are, respectively, equivalent to:
\begin{itemize}
\item[4''] No-twist condition (tangential part):
 A vector field $Z(t)$ is tangent parallel along the curve $\alpha(t)$ if, and only
if, $d_{\alpha(t)}g(t)(Z(t))$ is tangent parallel along $\widehat\alpha(t)$.

\item[5''] No-twist condition (normal part):
 A vector field $\Psi (t)$ is normal parallel along the curve $\alpha(t)$ if, and only
if, $d_{\alpha(t)}g(t)(\Psi(t))$ is normal parallel along $\widehat\alpha (t)$.
\end{itemize}
\end{proposition}

\begin{proof}
We prove the tangential part only. The proof of the normal part can be done similarly.

 It is clear that $\frac{D}{dt}Z (t)=0$ if, and only if, $\frac{D}{dt}\left(d_{\alpha(t)}g(t)(Z(t))\right)=0$. Consequently, condition 4 of Definition \ref{def_extrinsic} implies condition 4'' above.

To prove that condition 4'' implies condition 4 of Definition \ref{def_extrinsic}, let $Z (t)$ be an arbitrary tangent vector field along  $\alpha(t)$ and $\{E_1(t),\ldots, E_n(t)\} $, $n=\dim (M)$, be a parallel tangent frame along  $\alpha(t)$, so that
$$Z (t)=\sum_{i=1}^nz_i(t)E_i(t)\quad \text{and}\quad \frac{D}{dt}Z (t)=\sum_{i=1}^n\dot{z}_i(t)E_i(t).$$
Now define $\widehat{E}_i(t):=d_{\alpha(t)}g(t)(E_i(t))$. Taking into account assumption $4'$, we can guarantee that $\{\widehat{E}_1(t),\ldots, \widehat{E}_n(t)\} $ is a parallel tangent frame along the development curve $\widehat\alpha(t)$. Since $d_{\alpha (t)}g(t)$ is a linear isomorphism, using properties of the covariant derivative we obtain $$d_{\alpha(t)}g(t)\left(\frac{D}{dt}Z (t)\right)=\sum_{i=1}^n\dot{z}_i(t)d_{\alpha(t)}g(t)(E_i(t))=
\sum_{i=1}^n\dot{z}_i(t)\widehat{E}_i(t)$$
and
$$\frac{D}{dt}\left(d_{\alpha(t)}g(t)(Z(t))\right)=\frac{D}{dt}\left(
\sum_{i=1}^nz_i(t)\widehat{E}_i(t)\right)=\sum_{i=1}^n\dot{z}_i(t)\widehat{E}_i(t).$$
Therefore, condition 4 in Definition~\ref{def_extrinsic} follows.
\end{proof}

 \begin{remark} \label{codim} It is clear from  Proposition~\ref{prop-notwist}  that the tangent part of the  no-twist condition is always satisfied when the manifolds $M$ and $\widehat{M}$ are one-dimensional, and the normal no-twist condition is always satisfied when those manifolds have co-dimension one.
\end{remark}

In order to relate the two seemingly very different definitions of rolling when both $M$ and $\M$ are isometrically embedded in the semi-Riemannian vector space $(V,(.\,,.)_{p,n-p})$, we also need to compare $A(t)$, the part responsible for the isometric mapping of the tangent spaces in the intrinsic rolling definition, with the rolling map $g(t)=(R(t),s(t))\in \SE(V)$, in the extrinsic definition. 

 Since $\SE(V)=\SO(V)\ltimes V$ and $\SO(V)$ both act on $V$, if $\alpha (t)$ is a curve in $V$, then, for any $
g(t)=(R(t),s(t))\in \SE(V)$ and  any tangent vector field $Z(t)$ along $\alpha (t)$  we have 
\begin{equation}\label{new1}
d_{\alpha(t)}g(t).Z(t)= d_{\alpha(t)}R(t)Z(t) =R(t).Z(t).
\end{equation}
\begin{remark}\label{new_rem}
According to \emph{(\ref{new1})}, we can refer to the restriction of $d_{\alpha(t)}g(t)$ to $T_{\alpha(t)} M$, which is the same as the restriction of $d_{\alpha(t)}R(t)$ to $T_{\alpha(t)} M$, as the restriction of $R(t)$ to $T_{\alpha(t)} M$. This abuse of terminology simplifies the exposition that follows. 
\end{remark}

The following proposition provides  a relationship between the intrinsic and  the extrinsic rolling when $M$ and $\M$ are isometrically embedded in $V$.
\begin{proposition} \label{prop4}Assume that $(\alpha(t),\widehat{\alpha}(t),  A(t) )$ is a rolling curve in the sense of Definition \ref{def_intrinsic}. Let $g(t)=(R(t),s(t)) $ be a curve in $SE(V)$ such that the restriction of $R(t)$ to $T_{\alpha(t)}M$  is equal to $A(t)$ and $s(t)=\widehat{\alpha}(t)-R(t).\alpha(t)$. Then $g(t)$ satisfies  conditions 1 through 4 and 6 of Definition \ref{def_extrinsic}.

Conversely, if $g(t)=(R(t),s(t))$ is a curve in $SE(V)$ that satisfies conditions 1 through 4 and 6 of Definition \ref{def_extrinsic}, then $(\alpha(t), \widehat{\alpha}(t),A(t))$ is an intrinsic  rolling curve, where $A(t)$  is the restriction of $R(t)$ to $T_{\alpha(t)}M$. This happens, in particular, if $g(t)$ is a rolling map along $\alpha(t)$.
\end{proposition}
The previous statement is completely obvious in view of Proposition \ref{prop-notwist},  since the tangential no-twist condition in Definition \ref{def_extrinsic} is equivalent to the parallel transport condition required by the intrinsic rolling. According to the last statement of  Proposition \ref{prop4}, if $g(t)=(R(t),s(t))$ is a rolling map along $\alpha(t)$ with development $\widehat{\alpha}(t)$, we say that $R(t)$ defines the intrinsic  rolling curve $(\alpha(t), \widehat{\alpha}(t),A(t))$, where $A(t):= R(t)|_{T_{\alpha(t)}M}$.

 \begin{remark}We now also see precisely the   difference between  the rolling of Definition \ref{def_intrinsic} and the rolling of Definition~\ref{def_extrinsic}. A curve $\alpha(t)$ in $M$ rolls on a curve $\widehat{\alpha}(t)$  in $\M$ independently of the definition used.   However,  in the absence of the normal no-twist condition,  the lifting of the isometry $A(t)$ to an isometry $d_{\alpha(t)}g(t)$ in $T_{\alpha(t)}V$ is not one to one since the latter can be completely arbitrary on the orthogonal complement $(T_{\alpha(t)}M)^\perp$.  If $A^\perp (t): \, (T_{\alpha(t)}M)^\perp \rightarrow (T_{\widehat\alpha(t)}\widehat M)^\perp $ is a map such that any normal parallel vector field along $\alpha(t)$ maps to a normal parallel vector field along $\widehat\alpha(t)$, then $g(t)$ is completely and uniquely defined by
 $$
 d_{\alpha(t)}g(t)|_{T_{\alpha (t)}M}=A(t)\quad \text{and}\quad d_{\alpha(t)}g(t)|_{(T_{\alpha (t)}M)^\perp}=A^\perp (t).
 $$
  The arbitrariness of $A(t)$, which due to Remark \ref{new_rem} can be seen as an arbitrariness of $R(t)$,  is removed by adding the normal part of the no-twist condition, for then   there is a one to one correspondence   between $A(t)$  that rolls $\alpha(t)$ onto $\widehat{\alpha}(t)$ in the sense of Definition \ref{def_intrinsic}  and the  rolling map $g(t)=(R(t),s(t))$ in $SE(V)$ that rolls $\alpha(t)$  onto $\widehat{\alpha}(t)$ in the sense of Definition \ref{def_extrinsic}. In fact, $d_{\alpha(t)}(g(t))$ is equal to $A(t)$ on the tangent space $T_{\alpha(t)}M$, and is uniquely determined on the orthogonal complement by the normal no-twist condition.
\end{remark}
We have already seen in Proposition \ref{prop4} that if $g(t)=(R(t),s(t))$ is a rolling map  along a  curve $\alpha (t)$ with development $\widehat{\alpha}(t)$, then the triple $(\alpha (t), \widehat{\alpha}(t), A(t))$, where for $A(t):=R(t)|_{T_{\alpha (t)}M}$, is an intrinsically rolling curve. In other words, each extrinsic rolling map determines a unique intrinsic rolling curve.

 However, we may perturb $R(t)$ so that the normal part of the no-twist condition is violated and still obtain an intrinsic rolling curve of $\alpha (t)$ on $\widehat{\alpha}(t)$. The next proposition makes this statement clear. We use the symbol $\circ$ to denote the composition of linear maps. 

\begin{proposition} \label{RRtil} Suppose that $g(t)=(R(t),s(t))$ is a rolling map along the curve $\alpha(t)$ with development $\widehat{\alpha}(t)$, and $\dot R (t)=\Omega(t)\circ R(t)$, with $\Omega(t)\in\so(V)$. Let   
 $\widetilde R(t)$ be the solution of \be\label {inhom}
\dot{\widetilde R}(t)=(\Omega(t)+\Omega _0(t))\circ\widetilde R(t),\quad \widetilde R(0)=R(0), \ee
where $\Omega_0(t)\in\so(V)$ satisfies 
\be \label{omega0}
\Omega_0(t)(T_{\widehat{\alpha}(t)}\M)=0,\quad \Omega_0(t)(T_{\widehat{\alpha}(t)}\M)^\bot\subseteq (T_{\widehat{\alpha}(t)}\M)^\bot.
\ee
 Then $R(t)$ and $\widetilde R(t)$ define the same intrinsic rolling $(\alpha (t), \widehat{\alpha}(t), A(t))$.
\end{proposition}
\begin{proof} 
We already know that $R(t)$ defines the intrinsic rolling curve $(\alpha (t), \widehat{\alpha}(t), A(t))$, where $A(t)=R(t)|_{T_{\alpha(t)}M}$. In order to prove the statement it is enough to show that  $\widetilde R(t)|_{T_{\alpha(t)}M}=A(t) $. For that, first  notice that $R$ and $\widetilde R$ are related by
$\widetilde R(t)=R(t)\circ S(t)$, where  $S(t)\in SO(V)$ is the solution of
\be
\dot S (t)=(R^{-1}(t)\circ \Omega_0(t)\circ R(t))\circ S(t), \quad S(0)=I.
\ee 
Indeed, since $S(t)\in SO(V)$, we have $\dot S (t)=\Lambda(t)\circ S(t)$, for some $\Lambda(t)\in\so(V)$. And so, $$ \dot{\widetilde R}(t)=\dot R (t)\circ S(t)+R(t)\circ \dot S (t)= (\Omega(t)+R(t)\circ \Lambda (t)\circ  R^{-1}(t))\circ \widetilde R (t),$$
which according to the assumption $\dot{\widetilde R}(t)=(\Omega(t)+\Omega _0(t))\circ \widetilde R(t)$ implies $\Lambda (t)=R^{-1}(t)\circ  \Omega _0(t)\circ R(t)$.
Since by assumption $R(t)$ satisfies 
\be 
\begin{array}{l}
R(t)(T_{\alpha(t)}M) = (T_{\widehat{\alpha}(t)} \widehat M), \,\, \text{and}\,\,
R(t)(T_{\alpha(t)}M)^\bot = (T_{\widehat{\alpha}(t)} M)^\bot,
\end{array}
\ee
and 
$\Omega_0(t)$  satisfies (\ref{omega0}),  we conclude that for $\Lambda(t):=R^{-1}(t)\circ \Omega_0(t)\circ R(t)$,  
$$
\begin{array}{llll} 
\Lambda (t)(T_{\alpha (t)}M)
&=&
R^{-1}(t)\circ \Omega_0(t)(T_{\widehat{\alpha}(t)}\M)=0,
\\
\\
\Lambda (t)(T_{\alpha (t)}M)^\perp
&=&
R^{-1}(t)\circ \Omega_0(t)(T_{\widehat{\alpha}(t)}\M)^\perp\subseteq R^{-1}(t)(T_{\widehat{\alpha}(t)}\M)^\perp
\\
&=&
(T_{\alpha(t)}M)^\perp.\end{array}
$$
Now we are going to choose  a system of coordinates so that $S(t)|_{T_{\alpha(t)}M}$ becomes the identity map; that is $S(t)Z(t)=Z(t)$, for every $Z(t)\in T_{\alpha(t)}M$.
Let $\{ b_1(t),\dots,b_N(t)\}$ be an orthonormal frame  in $V$ along $\alpha (t)$ such that, for every $t$, $\{ b_1(t),\dots b_n(t)\}$ is a basis for $T_{\alpha(t)}M$ and $\{ b_{n+1}(t),\dots,b_N(t)\}$ is a basis for $(T_{\alpha(t)}M)^\perp$. In this system of coordinates  $\Lambda(t)(T_{\alpha(t)}M)$ is represented by the block matrix $\bpm 0_{n,n}\\0_{N-n,n}\epm$,  where $0_{m,n}$ denotes the zero matrix of size $m\times n$, while $\Lambda(t)(T_{\alpha(t)}M)^\perp$ is represented by the block matrix $\bpm 0_{n,N-n}\\\Lambda_0(t)\epm $,  where $\Lambda_0(t)$  is the projection of $\Lambda(t)(T_{\alpha(t)}M)^\perp$ on $(T_{\alpha(t)}M)^\perp$. As a consequence, $\Lambda (t)$ is represented by the block matrix
$$
\Lambda (t)=\bpm 0_{n,n}&0_{n,N-n}\\0_{N-n,n}&\Lambda_0(t)\epm.
$$
So, since $\dot{S}(t)=\Lambda (t)\circ S(t)$, we must have $S(t)=\bpm I_n&0_{n,N-n}\\0_{N-n,n}&S_4(t)\epm$, from what follows that  $S(t)Z(t)=Z(t)$, for every $Z(t)\in T_{\alpha(t)}M$, and, consequently, 
\bean 
\widetilde R(t)|_{T_{\alpha(t)}M}= R(t)|_{T_{\alpha(t)}M}=A(t).
\ean
\end{proof}

These subtle differences between various  notions of rollings are best illustrated  through the comparison of the rolling of a semi-Riemannian manifold $M$ over a flat manifold $\M$, versus the rolling of $M$ on its affine tangent space $\M$ at a fixed point, when $M$ and $\M$ are isometrically embedded in $(V,(.\,,.)_{p,n-p})$. We start by revising  general facts about a rolling on affine tangent spaces.
 
 The following important properties of the rolling map of Riemannian manifolds have also been proved in~\cite{Sh}. The proof uses the arguments involving the group properties of $SE(V)$, that are also true for the semi-Riemannian vector space $(V,(.\,,.)_{p,n-p})$.

\begin{proposition}\label{prop_rolling}
Let $M$, $M_1$ and $M_2$ be manifolds of the same dimension, isometrically embedded in $V$, and  $\alpha(t)$, $\alpha_1(t) $ and $\alpha_2(t)$  curves  in $M$, $M_1$ and $M_2$ respectively, defined in the real interval I, that satisfy $\alpha(0)=\alpha_1(0)=\alpha_2(0).$
\begin{itemize} 
\item  Symmetric property of rolling

 \noindent If $g(t)\in \SE(V)$ is a rolling map of $M$ on $M_1$ along the rolling curve $\alpha (t)$ with development curve $\alpha_1 (t)$, then
 $g^{-1}(t)\in \SE(V)$ is a rolling map of $M_1$ on $M$, along the rolling curve $\alpha_1(t)$ with development curve $\alpha(t)$.
\item Transitive property of rolling

 \noindent  If $g(t)\in \SE(V)$ is a rolling map of $M$ on $M_1$ along the rolling curve $\alpha (t)$ with development curve $\alpha_1 (t)$, and 
$g_1(t)\in \SE(V)$ is a rolling map of $M_1$ on $M_2$ along the rolling curve $\alpha_1(t)$ with development curve $\alpha_2(t)$, then
 $g(t) g_1 (t)\in \SE(V)$, is a rolling map of $M$ on $M_2$, with rolling curve $\alpha(t)$ and development curve $\alpha_2(t)$.
\end{itemize}
\end{proposition}
 \begin{remark}\label{rem-affine}
Using these two properties, one can reduce the study of rolling a manifold on another  to the simpler situation when the second manifold is the affine tangent space at a point of the first. Properties above have been used in~\cite{LLouro} to derive the kinematic equations of a sphere rolling on another sphere of the same dimension, using the equations of spheres rolling on  affine tangent spaces at a point. Also in the semi-Riemannian case, these properties have been used in~\cite{ML2022} to derive the kinematic equations for rolling a hyperbolic sphere over another. 
\end{remark}

 \section{Rolling of symmetric spaces on flat manifolds}\label{sec:RolSymSp}
 
 
 We start from setting the notation and recalling useful information about symmetric spaces based on~\cite{ONeill}.
 

 \subsection {Symmetric spaces}\label{sec:symmetric}

\begin{definition}
A connected semi-Riemannian manifold $\big(M,g\big)$ is called a semi-Riemannian symmetric space if for each $o\in M$ there exists a diffeomorphic isometric map $\zeta_o\colon M\to M$, called the global isometry of $M$ at $o$, such that $ d_o\zeta_o=-\Id$ on the vector space $T_oM$. 
\end{definition}

The symmetric semi-Riemannian spaces have close relation to Lie groups. The connected identity componnet $G$ of the isometry group acts transitively on $M$. Let $H$ be the isotropy subgroup of a point $o\in M$. Then $M$ can be identified with the homogeneous space $G/H$. Note that the isotropy subgroups of different points are conjugate subgroups of $G$. Let $\fg$ and $\fh$ be the Lie algebras of the Lie groups $G$ and $H$ respectively. Then, the following Cartan decomposition holds, 
\begin{equation}\label{eq:Cartan}
\fg=\fh\oplus \fp,\quad [\fh,\fh]\subset\fh,\quad [\fh,\fp]\subset\fp, \quad[\fp,\fp]\subset\fh.
\end{equation}
We denote by $\tau\colon G\times M\to M$, $(q,m)\mapsto  \tau (q,m)=q.m$ the action of $G$ on $M$. Then, for any fixed $q\in G$, $\tau_q\colon M\to M$ is a diffeomorphism of $M$. Recall that a metric tensor $g(.\,,.)$ on $M$  is said to be \emph{$G$-invariant} if 
$$
g(X(m),Y(m))=g(d_m\tau_q(X),d_m\tau_q(Y)), 
$$
for $q\in G$, and vector fields $X,Y$ on $M$. A scalar product $\langle.\,,.\rangle$ in $\fg$ is said to be \emph{ $\Ad_H$-invariant} if 
$$
\langle \Ad_h X,\Ad_hY\rangle=\langle X,Y\rangle,\quad h\in H,\quad X,Y\in \fg.
$$

Let $\pi$ denote the projection of $G$ on the coset manifold, i.e., $\pi\colon G\to G/H=M$, $g\mapsto g.o=m$ . If $e$ is the identity of $G$, then the map $\pi$ and the differential map
\be \label{diffmap}
d_e\pi\colon T_eG=\fg\to  T_oM
\ee
have the following properties, see~\cite[Chapter 11]{ONeill}:
\begin{enumerate}
\item[\emph{1.}]  $\pi\colon G\to G/H=M$ is a submersion, such that $d_e\pi(\fh)=\{0\}\subset T_oM$, and 
$d_e\pi\colon \fp\to T_oM$ is an isomorphism;
\item[\emph{2.}] $d_e\pi$ makes one-to-one correspondence between $Ad_H$-invariant scalar products on $\fp$ and $G$-invariant metrics on $M$.
\end{enumerate}

\begin{definition}\label{def:semi-RiemannianSubmersion}
Let $(M,g^M)$ and $(N,g^N)$ be two semi-Riemannian mani\-folds and $\pi\colon N\to M$ a submersion such that $T_nN=\mathcal V_n\oplus\mathcal H_n$, with $\mathcal V_n=\ker(d_n\pi)$. Then $\pi$ is called a semi-Riemannian submersion if $\pi^{-1}(m)$ is a Riemannian submanifold of $N$ and the direct sum $\mathcal V_n\oplus\mathcal H_n$ is orthogonal at each $n\in N$.
\end{definition}

Let $M=G/H$ be a semi-Riemannian symmetric space with $G$-invariant metric $g(.\,,.)$ corresponding to an $\Ad_H$-invariant scalar pro\-duct  $\langle.\,,.\rangle$ on $\fp$,  as  mentioned in property~\emph{2} before Definition~\ref{def:semi-RiemannianSubmersion}. We extend $\langle.\,,.\rangle$ to the entire Lie algebra $\fg$ such that the direct sum $\fg=\fp\oplus\fh$ becomes orthogonal. We denote by $L_qh=qh$ the multiplication from the left on $G$. We then define the vertical left invariant distribution $\mathcal V$ by $\mathcal V_q=d_eL_q(\fh)$ and the horizontal distribution $\mathcal H$ by $\mathcal H_q=d_eL_q(\fp)$. We keep the notation $\langle.\,,.\rangle$ for the left-invariant $\Ad_H$-invariant metric on $G$ induced by the extended scalar product on $\fg$. Under these conditions the projection map $\pi\colon G\to M$ is a semi-Riemannian submersion. 

We say that a vector field $X$ on $G$ is horizontal if $X(q)\in\mathcal H_q$ for any $q\in G$. An absolutely continuous curve $q\colon I\to G$ on $G$ is horizontal if $\dot q(t)\in \mathcal H_{q(t)}$ for almost every $t\in I$, or equivalently, if $ q^{-1}(t)\dot q(t)\in 
\mathcal H_{q(0)}$.

For a vector field $Y$ on $M$ there is a horizontal vector field $\widetilde Y$ on $G$ such that $d_q\pi (\widetilde Y(q))=Y_{\pi(q)}$. In particular, this implies that for any absolutely continuous curve $\alpha\colon I\to M$ there is a horizontal curve $q\colon I\to G$ such that $\pi(q(t))=\alpha(t)$ and $d_{q(t)}\pi (\dot q(t))=\dot\alpha(t)$ for almost every $t\in I$. We call $\widetilde Y$ and $q(t)$ the horizontal lifts of $Y$ and $\alpha(t)$, respectively. A horizontal lift $q(t)$ of a curve $\alpha (t)$ is unique, if we specify the initial value $q(0)$.

The following result  will be useful later on.

 \begin{lemma}\label{lem:CovDer}
 Let $\pi\colon G\to M$ be a semi-Riemannian submersion onto a semi-Riemannian symmetric space as above. Let $\alpha\colon [0,T]\to M$ be an absolutely continuous curve and $Y$ be a vector field along $\alpha$. Let $q\colon [0,T]\to G$ be a horizontal lift of $\alpha$ and $\widetilde Y$ a horizontal lift of $Y$ along $q$. Then 
\begin{equation}\label{eq:NN0}
 \frac{D^M_{\alpha(t)}}{dt}Y(t)=d_{q(t)}\pi\Big(\sum_{j=1}^k\frac{dy_j(t)}{dt} A_j\Big),
\end{equation}
where $d_{q(t)}L_{q^{-1}(t)}\widetilde Y(t)=\sum_{j=1}^ky_j(t)A_j\, $ is written in terms of a left-invariant basis $\{A_1,\ldots,A_k\}$ of $\fp$.
\end{lemma}
\begin{proof}
We denote by $\nabla^G$ the Levi-Civita connection on $G$. First we show that 
\begin{equation}\label{eq:Nabla}
\nabla^G_VW=\frac{1}{2}[V,W]
\end{equation} 
for left invariant vector fields $V,W\in\fp$. Since $\fp$ and the metric on $G$ are $\Ad_{H}$-invariant, then
\begin{equation}\label{eq:ADHINV}
\langle [V,Z],W\rangle=\langle V,[Z,W]\rangle,\quad V,W\in\fp,\quad Z\in\fh,
\end{equation}
see, for instance~\cite[Lemma 3, Chapter 11]{ONeill}. Then for any $Z\in\fg$ and $V,W\in\fp$ we have 
\begin{equation}\label{eq:Koszul}
2\langle \nabla^G_VW,Z\rangle=-\langle V,[W,Z]\rangle+\langle W, [Z,V]\rangle+\langle Z,[V,W]\rangle
\end{equation}
by Koszul formula. If $Z\in\fh$, then the first two terms on the right hand side are cancelled by~\eqref{eq:ADHINV}. If $Z\in\fp$, then the first two terms on the right-hand side vanish by $[\fp,\fp]\in\fh$ and the orthogonality of $\fp$ and $\fh$. It shows~\eqref{eq:Nabla}.

 Let $\pi\colon G\to M$ be a Riemannian submersion, $X$, $Y$ vector fields on $M$, and $\widetilde X$, $\widetilde Y$ their horizontal lifts to $G$. We denote $\pr_{\mathcal H_q}\colon T_qG\to\mathcal H_q$ the orthogonal projection onto a horizontal sub-bundle $\mathcal H$ at $q\in G$. 
We recall that the Levi-Civita connections $\nabla^M$ on $M$ and $\nabla^G$ on $G$ are related by
\begin{equation}\label{eq:NN}
\nabla^M_XY=d_q\pi(\pr_{\mathcal H_q}\nabla^G_{\widetilde X}\widetilde Y),
\end{equation}
see~\cite[Lemma 45, Chapter 7]{ONeill}.
We write the horizontal lifts $\widetilde X$ and $\widetilde Y$, in the left invariant basis $\{A_1,\ldots,A_k\}$ of $\fp$ by
$$
d_{q(t)}L_{q^{-1}(t)} \widetilde X=\sum_{i=1}^kx_iA_i, \quad d_{q(t)}L_{q^{-1}(t)}\widetilde Y=\sum_{j=1}^ky_jA_j
$$
Then, 
\begin{eqnarray*}
\nabla^G_{\widetilde X} \widetilde Y&=&
\nabla^G_{\sum_{i=1}^kx_iA_i} \sum_{j=1}^ky_jA_j\, =\, 
\sum_{i,j=1}^kx_i\Big(\nabla^G_{A_i}y_jA_j\Big)
\\
&=&
\sum_{i,j=1}^kx_i\Big(A_i(y_j)A_j+y_j\nabla^G_{A_i}A_j\Big)
\\
&=&
\sum_{j=1}^k\Big(\sum_{i=1}^k(x_iA_i)(y_j)\Big)A_j+\sum_{i,j=1}^kx_iy_j\nabla^G_{A_i}A_j
\\
&=&
\sum_{j=1}^k\Big(\sum_{i=1}^k(x_iA_i)(y_j)\Big)A_j+\frac{1}{2}\sum_{i,j=1}^kx_iy_j[A_i,A_j].
\end{eqnarray*}
Since $M$ is a semi-Riemannian symmetric space we have $[A_i,A_j]\subset\fh$ and therefore
$$
\pr_{\mathcal H_q}\nabla^G_{\widetilde X}\widetilde Y=\sum_{j=1}^k\Big(\sum_{i=1}^k(x_i(q)A_i)(y_j(q))\Big)A_j .
$$
Now we set $d_{q(t)}L_{q^{-1}(t)}\widetilde X(t)=d_{q(t)}L_{q^{-1}(t)}\dot q(t)=\sum_{i=1}^kx_i(q(t))A_i$ and obtain
\begin{eqnarray*}
\pr_{\mathcal H_{q(t)}} \frac{D^G_{\dot q(t)}\widetilde Y(t)}{dt} 
&=&
\pr_{\mathcal H_{q(t)}} \nabla^G_{\dot q(t)} \widetilde Y(t)\, =\, 
\sum_{j=1}^k\Big(\dot q(t)(y_j(q(t)))\Big)A_j
\\
&=&
\sum_{j=1}^k\frac{dy_j(q(t))}{dt} A_j.
\end{eqnarray*}
We set $X(t)=\dot\alpha(t)$ in formula~\eqref{eq:NN} and obtain~\eqref{eq:NN0}.
\end{proof}


\subsection{Intrinsic rolling of symmetric spaces on flat manifolds}\label{RolSymSp}

The defi\-ni\-tion of a symmetric space is intimately related to the rolling on a flat space.  We aim to construct an intrinsic rolling of a semi-Riemannian symmetric manifold $M$ on the tangent space $T_oM=\M$. Namely, we will find the triplet $\big(\alpha(t),\widehat\alpha (t), A(t)\big)$ satisfying Definition~\ref{def_intrinsic} by using  only the data of the symmetric manifold. 

The main properties that result from assuming that $M$ is a symmetric space can be summarised in the following commutative diagrams, 

 \begin{equation}\label{eq:diag1}
\xymatrix
{
& G\ar[d]_-{\ds L_q} \ar[r]^-{\ds\pi}
& M\ar[d]^-{\ds\tau_q}
\\
 &G
 \ar[r]_-{\ds\pi}
 &M 
}\quad 
\xymatrix
{
& \fg= \fp\oplus \fh \ar[d]_-{\ds d_e(L_q)} \ar[r]^-{\ds d_e\pi}
& T_{o}M\ar[d]^-{\ds d_o(\tau_q)}
\\
&T_{q}G=\mathcal H_q\oplus\mathcal V_q 
 \ar[r]_-{\ds d_q\pi}
& T_{ \tau_q(o)}M
}
\end{equation}
Thus we conclude that
\begin{equation}\label{eq:pig}
\pi(q)=\pi(L_q(e))=\tau_{q}(\pi(e))=\tau_q(o),\quad q\in G.
\end{equation}

We also recall that, $\forall q\in G$,
$$\mathcal V_q=\ker(d_q\pi), \quad \fh=\ker(d_e\pi),
$$ 
and the map $d_q\pi\colon \mathcal H_q\to T_{\tau_q(o)}M$ is an isometry. Let $U\in\fp$. Then, from the second diagram in~\eqref{eq:diag1}, it follows that 
\begin{equation}\label{eq:dpig}
d_q\pi(d_eL_q(U))=d_o\tau_q(d_e\pi(U)).
\end{equation}

Now, choose an absolutely continuous curve $\alpha\colon [0,T]\to M$ such that $\alpha(0)=o$. Then, there exists  a horizontal curve $q\colon [0,T]\to G$ that projects to $\alpha$. More precisely
\begin{itemize}
\item[L1]{$\pi(q(t))=\alpha(t)=\tau_{q(t)}(o)$. In the last equality we used~\eqref{eq:pig};}
\item[L2]{$d_{q(t)}\pi(\dot q(t))=\dot\alpha(t)$;}
\item[L3]{$\dot q(t)=d_eL_{q(t)}U(t)$, for some curve $U\in\fp$.}
\end{itemize}
Combining $L3,L2$ and~\eqref{eq:dpig} we obtain
\begin{equation}\label{eq:dpig1}
\dot\alpha(t)\stackrel{L2}=d_{q(t)}\pi(\dot q(t))\stackrel{L3}=d_{q(t)}\pi(d_eL_{q(t)}U(t))\stackrel{\eqref{eq:dpig}}=d_0\tau_{q(t)}(d_e\pi(U(t))),
\end{equation}
and emphasize that both maps $d_{q(t)}\pi$ and $d_0\tau_{q(t)}$ are isometries between the corresponding spaces, for any $t\in [0,T]$. 

Now we define the curve  $\widehat\alpha (t)\in T_oM$. For the curve $U(t)=d_{q(t)}L_{q^{-1}(t)}(\dot q(t))$ on $\fp$, we write $d_e\pi(U(t))\in T_oM$,  and by solving the Cauchy problem
\begin{equation}\label{eq:CauchyPr}
\begin{cases}
\frac{d\widehat\alpha}{dt}=d_e\pi(U(t))=\dot{\widehat\alpha}
\\
\widehat\alpha(0)=0,
\end{cases}
\end{equation}
we obtain $\widehat\alpha (t)$
Here we implicitly identified the vector space $T_oM$ and $T_{\widehat\alpha(t)}(T_oM)$ by an isometric orientation preserving map $j$. 

We also define $A(t)\colon T_{\alpha(t)}M\to T_{\widehat\alpha(t)}(T_oM)$ as a composition of the maps
\begin{equation}\label{eq:isom}
T_{\alpha(t)}M\stackrel{L1}=T_{\tau_{q(t)}(o)}M\xrightarrow[\text{by}~\eqref{eq:diag1}]{ (d_0\tau_{q(t)})^{-1}} T_oM\xrightarrow[]{\,\,  j\,\, } T_{\widehat\alpha(t)}(T_oM).
\end{equation}
Note that, by the commutative diagram~\eqref{eq:diag1}, the map $A(t)$ can be defined alternatively  by the composition of the following isometric maps,
\begin{equation}\label{eq:isom1}
T_{\alpha(t)}M\xrightarrow[]{ (d_{q(t)}\pi)^{-1}} \mathcal H_{q(t)}\xrightarrow[]{ d_{q(t)}L_{q^{-1}(t)}}\fp\xrightarrow[]{\,  d_{e}\pi \,}T_oM\xrightarrow[]{\,  j\, } T_{\widehat\alpha(t)}(T_oM).
\end{equation}

%
%
%
%
%
%
%
%
\begin{proposition}\label{prop:rolling}
 If $\alpha (t)$, $\widehat\alpha (t) $ and $A (t)\colon T_{\alpha(t)}M\to T_{\widehat\alpha(t)}(T_oM)$ are  defined as in Section~\ref{RolSymSp}, the triple $(\alpha (t),\widehat\alpha (t), A(t))$ is a rolling curve for the intrinsic rolling of the manifold $M$ over $\widehat M=T_oM$, i.e., it satisfies conditions in  Definition \ref{def_intrinsic}.
\end{proposition}
\begin{proof}

By the construction of $A$, the condition~\eqref{iso} in the intrinsic rolling for the isometry $A(t)\colon T_{\alpha(t)}M\to T_{\widehat\alpha(t)}(T_oM)$  is fulfilled. We need to verify that a parallel vector field $Y$ along $\alpha$ is mapped to a parallel vector field $\widehat Y$ along $\widehat\alpha$. 

Let $Y$ be a vector field along $\alpha(t)=\tau_{q(t)}(o)$, where $q(t)$ is a horizontal lift of $\alpha(t)$. Then $\widehat Y=A(Y)$ is given by 
\begin{equation}\label{eq:VV}
\widehat Y(t)=d_e\pi\circ d_{q(t)}L_{q^{-1}(t)}(\widetilde Y(t))=d_e\pi\Big(\sum_{j=1}^k\widetilde y(t)A_j\Big),
\end{equation}
where we used~\eqref{eq:isom1}. We also denoted by $\widetilde Y(t)\in\mathcal H_{q(t)}$  the horizontal lift of $Y(t)$, and write $d_{q(t)}L_{q^{-1}(t)}\widetilde Y(t)=\sum_{j=1}^k\widetilde y(t)A_j$, where $\{A_1,\ldots,A_k\}$ is a basis for $\fp$. Assume that $Y(t)$ is a parallel vector field along $\alpha$. Then, using the identity \eqref{eq:NN0} in Lemma~\ref{lem:CovDer}, $d_{q(t)}\pi\Big(\sum_{j=1}^k\frac{d\widetilde y(t)}{dt}A_j\Big)=0$. Since $d_{q(t)}\pi\colon \mathcal H_{q(t)}\to T_{\alpha(t)}M$ is a bijection, we conclude that $\frac{d \widetilde y_j(t)}{dt}=0$ for all $j=1,\ldots,k$.  Then, \eqref{eq:VV} shows that 
$$
\frac{d\widehat Y(t)}{dt}=d_e\pi\Big(\sum_{j=1}^k\frac{d\widetilde y(t)}{dt}A_j\Big)=0.
$$
Thus $\widehat Y$ is a parallel vector field along $\widehat\alpha$ on $T_oM$.
\end{proof}



\subsection{Extrinsic rolling of symmetric spaces on flat manifolds}\label{RolSymSpExt}

In the present section we describe the rolling of a semi-Riemannian symmetric manifold $M$ on the  flat manifold which is the affine tangent space $\widehat M=T_{o}^{\text{aff}}M$ at $o\in M$. We will connect the intrinsic rolling, described in Section~\ref{RolSymSp} to the extrinsic rolling, by choosing an isometric embedding into a vector space $V$.  Let $M=G/H$ be a semi-Riemannian symmetric manifold.
Let  
\begin{equation}\label{eq:ext1}
\iota\colon M\to V,\quad \overline M=\iota(M)
\end{equation}
be an isometric embedding. In the present section all the objects related to $V$ will be marked by a line on top, like the image $\overline M=\iota(M)\subset V$ of $M$ in $V$, or $\overline o=\iota(o)$ the image of the isotropy point in $V$. The map $d_o\iota$ is a linear isometry and
\begin{equation}\label{eq:ext2}
d_o\iota(T_oM)=T_{\iota(o)}\overline M=T_{\overline o}\overline M\subset V.
\end{equation}
We define 
\begin{equation}\label{eq:ext3}
\overline{\widehat M}:=T^{\text{aff}}_{\overline o}\overline M=\overline o+T_{\overline o}\overline M.
\end{equation}
Note  that diagram~\eqref{eq:diag1} implies that any $\overline W\in T_{\overline o}\overline M$ can be written as
\begin{equation}\label{eq:ext4}
\overline W=d_o\iota(W)=d_o\iota(d_e\pi(U)),\quad \mbox{where} \quad W\in T_oM,\,\, U\in\fp.
\end{equation}
 We assume that $\rho\colon G\to \GL(V)$  is a representation of $G$ on $V$, and define
\begin{equation}\label{eq:ext5}
\overline G:=\rho(G).
\end{equation}

The action of $\overline G$ on $\overline M$ is denoted by $\overline q.\overline m$, with $\overline q=\rho(q)\in\overline G$ and $\overline m\in V$, to emphasize that the group $\overline G$ acts on both $\overline M$ and $\overline{\widehat M}$ as it does on vectors in $V$. We keep writing $\tau_q$ for the action of $q\in G$ on $M$. Moreover, we assume that the imbedding map $\iota$ is \emph{equivariant} under these actions, i.e.,
\begin{equation}\label{eq:ext5a}
\iota(\tau_q(m))=\overline q.(\iota(m))=\overline q.\overline m.
\end{equation}

We know that $G$ acts on the symmetric space $M$ by isometries and $\iota: M\rightarrow V$ is an isometric embedding. So, since the metric on $\overline M$ is the restriction of the metric on $V$,  $\overline G=\rho (G)$ must preserve the metric on $V$. As a consequence, $\overline G\subset \SO(V)$.

The group representation $\rho$ induces the Lie algebra representation $d_e\rho$ that maps $A\in \fg$ to $\bar A \in \bar\fg\subset\so(V)$. Let $(.\,, .)$ denote the scalar product on $\bar \fg$, defined by 
$$
(\bar A, \bar B):=\langle A,B \rangle.
$$
Then $\bar\fp =d_e\rho (\fp)$ is the orthogonal complement to $\bar\fh =d_e\rho (\fh)$ in $\bar\fg$ relative to $(.\,,.)$, and $\bar\fg =\bar\fh \oplus \bar\fp$ is a Cartan decomposition. 

Define the map $\mathbb{P}: \overline G \rightarrow \overline M$ by
\begin{equation}\label{eq:ext5b}
\mathbb{P}(\rho (q))=\iota(\pi (q)).
\end{equation}
The map $\mathbb{P}$ is smooth, as a composition of smooth maps.

 Differentiating~\eqref{eq:ext5a}, we get that the lower part of the following diagram commutes, while differentia\-ting~\eqref{eq:ext5b} we also get that the upper part of this diagram commutes.
\begin{equation}\label{eq:ext6}
\xymatrix
{
& \fp \ar[d]_-{\ds d_e\pi} \ar[r]^-{\ds d_e\rho}
& \overline\fp
\ar[d]^-{\ds p=d_e\mathbb P}
\\
&T_{o}M 
 \ar[d]_-{\ds d_o\tau_q}
 \ar[r]^-{\ds d_o\iota}
& T_{\overline o}\overline M
 \ar[d]^-{\ds d_{\overline o}\overline q}
\\
&T_{\tau_q(o)}M
\ar[r]^-{\ds d_{\tau_q(o)}\iota}
&T_{\overline q.\overline o}\overline M
}
\end{equation}
Here  all the linear maps are bijective isometries. 

We take an absolutely continuous curve $\alpha\colon [0,T]\to M$,  $\alpha(0) = o$, and its horizontal lift 
\begin{equation}\label{eq:HorizLift}
q\colon [0,T]\to G,\quad  d_{\dot q(t)}L_{q^{-1}(t)}\dot q(t)=U(t)\in\fp,
\end{equation}
such that $\alpha(t)=\tau_{q(t)}(o)$. This and the equivariance of $\iota$ given by~\eqref{eq:ext5a} imply that
\begin{equation}\label{eq:ext7}
\overline \alpha(t) =\iota(\alpha(t))=\iota(\tau_{q(t)}(o))=\overline q (t).\overline o ,
\end{equation}
where $\overline q\colon [0,T]\to \overline G$ is horizontal, i.e., $\bar q^{-1}\dot{\overline q}\in \bar\fp $.
But then,
\begin{eqnarray}\label{eq:ext8}
\dot{\overline \alpha}(t) &=&d_{\alpha(t)}\iota(\dot\alpha(t))=d_{\alpha(t)}\iota\circ d_o\tau _q(t)( d_e\pi (U))\nonumber
\\
&=&
d_{\overline o}\overline q (t)(p\circ  d_e\rho(U))=d_{\overline o}\overline q(t)(\overline W)
\end{eqnarray}
where $U\in\fp$ from~\eqref{eq:HorizLift}, and  $\overline W:=p\, \circ \, d_e\rho(U)\in T_{\overline o}\overline M$, by diagram~\eqref{eq:ext6}. Since $\overline q(t)$ is a linear map, then $d_{\overline o}\overline q(t)=\overline q(t)$ and we simply write $\overline W(t)=\overline q(t)^{-1}.\, \dot{\overline\alpha}(t)$ for $\overline q(t)$ from~\eqref{eq:ext7}. 

We are ready to define $\overline{\widehat \alpha}\in \overline{\widehat M}=T^{\text{aff}}_{\overline o}\overline M$. For that we find a curve $\overline s\colon [0,T]\to T_{\overline o}\overline M$ as the solution of the Cauchy problem
\begin{equation}\label{eq:CauchyPrExt}
\begin{cases}
\dot{\overline s}(t)=\overline W(t)=\overline q(t)^{-1}.\dot{\overline\alpha}(t)
\\
\overline s(0)=0 
\end{cases},
\end{equation}
and set $\overline{\widehat \alpha} (t)=\overline o+\overline s (t)\in T^{\text{aff}}_{\overline o}\overline M$. 
\begin{proposition}\label{prop:rollingExt}
In the notation of Section~\ref{RolSymSpExt}, there is $\overline R\colon [0,T]\to \SO(V)$ such that  $\overline g(t)=(\overline R(t), \overline s (t))\in \SE(V)$ is a rolling map that rolls the curve $\overline \alpha (t)\in\overline M$ onto the curve $\overline{\widehat \alpha} (t)\in \overline{\widehat M}=T^{\text{aff}}_{\overline o}\overline M$, 
where $\overline{\widehat \alpha} (t)=\overline o+\overline s (t)$.
\end{proposition}
We emphasise that for the rolling map $\overline g(t)$ in the statement of this proposition, one has the freedom to define $\overline R(t)\vert_{T^{\perp}\overline M}$ such that $g(t)$ satisfies the normal no-twist condition. The no-slip and tangent no-twist conditions are determined by the intrinsic rolling, the latter being the nature of symmetric spaces.

\begin{proof}
The proof is constructive. Since $\overline s$ is the solution of~\eqref{eq:CauchyPrExt}, it is enough to define $\overline R(t)\in\SO(V)$ so that $\overline g(t)=(\overline R(t), \overline s (t))\in \SE(V)$ satisfies the conditions in Definition~\ref{def_extrinsic}.

Let $\overline R (t)$ be such that
\begin{equation}\label{eq:ext9new}
d_{\overline\alpha(t)}\overline R (t)\vert_{T_{\overline\alpha(t)}\overline M}(\overline X (t))=\overline q(t)^{-1}.\, \overline X (t),
\end{equation}
 for any tangent vector field $\overline X (t)$ along $\overline\alpha(t)$ and $\overline R (0)=e$.
 
\noindent Using (\ref{eq:ext8}) and (\ref{eq:CauchyPrExt}), we see that $g(t)$ satisfies the no-slip condition
$$
d_{\overline\alpha(t)}\overline g (t) (\dot{\overline\alpha }(t))= \overline q(t)^{-1}.\, \dot{\overline\alpha }(t)=\dot{\overline s }(t)=\dot{\overline{\widehat\alpha }}(t).
$$

Now, we show that if $\overline R(t)$ is defined as above, then $\overline g(t)=(\overline R(t), \overline s (t))$  satisfies the tangent no-twist condition given in Proposition
\ref{prop-notwist}. Let $\overline X(t)$ be a tangent parallel vector field along $\overline{\alpha}(t)$. Notice that, from the bottom of diagram~\eqref{eq:ext6}, $d_{\alpha}\iota$  is a bijective isometry between 
$T_{\alpha}M$ and $T_{\overline \alpha}\overline M$. Since, according to~\cite[Proposition 3.59]{ONeill}, the covariant derivative on $M$ is the pullback of the covariant derivative on $\overline M$ under the isometries, 
the vector field $X(t)=\big(d_{\alpha(t)}\iota\big)^{-1}(\overline X(t))$ is  parallel along $\alpha (t)$ on $M$.

 Moreover, diagram~\eqref{eq:ext6} shows that the parallel vector field $X(t)$ is  mapped to the parallel tangent vector field $\overline{\widetilde X}(t)$ along $\overline s(t)$ on $T_{\overline o}\overline M$ due to Lemma~\ref{lem:CovDer} and the isometric embedding. As a consequence, the vector field $\overline{\widehat X}(t)$ along $\overline{\widehat \alpha}$ on $\overline{\widehat M}=T^{\text{aff}}_{\overline o}\overline M$ is parallel.


As was mentioned earlier on, the condition (\ref{eq:ext9new}) on $\overline R(t)$  still leaves freedom on how $d_{\overline\alpha(t)}\overline R (t)$ acts on the normal space $T_{\overline\alpha(t)}^{\perp}\overline M$. 
In order to guarantee that  $\overline g(t)=(\overline R(t), \overline s (t))\subset \SE(V)$ also satisfies the normal no-twist condition, we define the (unique) map $\overline R(t)$ along $\overline{\alpha}$ on $\overline M$ such that the differential $d_{\overline \alpha(t)}\overline R(t)\vert_{T^{\perp}_{\overline\alpha(t)}\overline M}$
 maps the normal parallel vector fields along $\overline{\alpha}$ to the normal parallel vector fields along $\overline{\widehat \alpha}$ on $\overline{\widehat M}=T^{\text{aff}}_{\overline o}\overline M$.
\end{proof}

\begin{remark}  In relation to the last part of the proof of Proposition \ref{prop:rollingExt}, we point out that in~\cite[Section 3.3]{MOL} a complete answer was given to the problem of extending intrinsic rollings to extrinsic ones. We also refer to \cite{KML} for non-twist conditions in the case of embedded sub-Euclidean manifolds.
\end{remark}

\begin{corollary} \label{cor-codim}
If   $\overline M$ has co-dimension $1$, let $\overline \alpha(t) =\overline q (t).\overline o$ be  a curve in $\overline M$, satisfying $\overline \alpha(0) =\overline o$, where $\overline q (t)$ is a horizontal curve in $\overline G$  and $\dot{\overline q}= \overline q \, .\, \overline U(t)$. Then,
$(\overline R(t),\overline s (t))\in \SE(V)$ is a rolling map of $\overline M$ on $\overline{\widehat M}$ along  $\overline \alpha(t)$, with development $\overline {\widehat\alpha}(t)=\overline s (t)+\overline o$, where $\overline R(t)=\overline q(t)^{-1}$  and $\overline s (t)$ satisfies the Cauchy problem $\dot{\overline s} (t)=\overline U(t).\overline o$, $\overline s(0)=0$. Moreover,
\be
\left\{\begin{array}{l}
\dot{\overline R} (t)=- \overline U(t).\overline R (t)\\
\dot{\overline s} (t)=\overline U(t).\overline o
\end{array}, \right.
\ee
are the corresponding kinematic equations. 
\end{corollary} 
\begin{proof}
For manifolds of co-dimension $1$, the normal non-twist condition is always satisfied. So, taking into account (\ref{eq:ext9new}) and Remark \ref{new_rem}, we have $\overline R (t)=\overline q(t)^{-1}$. So, $\dot{\overline R} (t)=- \overline U(t)\overline R (t)$. According to (\ref{eq:CauchyPrExt}), $\dot{\overline s} (t)=\overline q(t)^{-1}.\dot{\overline \alpha}(t)$. But 
$\dot{\overline \alpha}(t) =\dot{\overline q} (t).\overline o$, so it follows that
$\dot{\overline s} (t)=\overline q(t)^{-1}\dot{\overline q} (t).\overline o=\overline U(t).\overline o$.
\end{proof}


\subsection{Examples}


We will exemplify the results of Sections~\ref{RolSymSp} and~\ref{RolSymSpExt}. 


\subsubsection{Rolling the $2$-dimensional hyperbolic space.}
 

First we describe the hyperbolic disc as a symmetric manifold, and construct the intrinsic and extrinsic rolling on the corresponding flat spaces.
We refer to \cite{L} for more details about hyperbolic spaces and the relationship between two of its equivalent models, which will be used in this section.
  
 Let $\cD$ be the unit disk $\{z\in\IC:|z|<1\}$ in $\R ^2$, with the hyperbolic metric given in coordinates $(x_1,x_2)$ by $h^2=4\frac{(dx_1)^2+(dx_2)^2}{\left( 1-(x_1^2+x_2^2)\right) ^2}$. $\cD$ is also known as the Poincar\'e ball model. The Lie group 
 $$G=SU(1,1)=\left\{ g=\bpm a&b\\\bar b&\bar a\epm:\,\, a,b\in \IC,\,\, |a|^2-|b|^2=1\right\}$$ acts transitively on  $\cD$ via the M\"oebius transformations, i.e., 
\bean
\tau_g (z)=\frac{az+b}{\bar b z+\bar a}, \quad \tau_g(0)=\frac{b}{\bar a}.\ean  
Let  $H:=\left\{\bpm a&0\\0&\bar a\epm,|a|^2=1\right\}$ be the isotropy subgroup of $0\in \cD$.  The projection map is
\be \label{projmap}\begin{array}{lccc}
\pi:&G&\rightarrow & \cD=G/H \\
&g=\bpm a&b\\\bar b&\bar a\epm&\mapsto &\tau_g(0)=\frac{b}{\bar a}
\end{array}
\ee
The Lie algebra $\fg$ of $G$ is given by 
$$
\fg=su(1,1)=\left\{ \bpm iv&u_1+iu_2\\ u_1-iu_2& -iv\epm:\,\, v,u_1,u_2\in\IR\right\}.
$$
We endow $\fg$ with  an $ Ad_G$-invariant  semi-Riemannian metric  defined by $\langle X,Y\rangle =2\tr(XY)=\frac{1}{2}B(X,Y)$, \, $B(.\,,.)$  being the Killing form.
The matrices 
\be\label{pauli}
A_1=\frac{1}{2}\bpm i&0\\0&-i\epm, A_2=\frac{1}{2}\bpm 0&1\\1&0\epm, A_3=\frac{1}{2}\bpm 0&i\\-i&0\epm 
\ee form an orthonormal basis of $\fg$.

The Lie algebra $\fh$ of the isotropy subgroup $H$ is spanned by $A_1$ and its orthogonal complement $\fp$ is spanned by $A_2$ and $A_3$. Note that the restriction of $\langle .\,,.\rangle$ to $\fp$ is positive definite. From the commutation relations, we conclude that $\fg=\fh\oplus \fp$ is a Cartan decomposition of $\fg$.

A curve $z(t)$ in $\cD$ lifts to a horizontal curve 
$$
g(t)=\frac{1}{\sqrt{1-|z(t)|^2}}\bpm 1&z(t)\\\bar z(t)&1\epm e^{\theta (t) A_1}\in G,
$$ 
when $\dot\theta=\frac{-2}{1-|z|^2}(x_1\dot x_2-\dot x_1x_2)$. In such  case,
\be \label{g-1dotg}
g^{-1}\dot g =\frac{1}{1-|z|^2}\bpm 0&\dot z\, e^{-i\theta}\\\dot{\bar{ z}}\, e^{i\theta}&0 \epm . \ee
The proof of these two facts regarding lifts of curves can be found in~\cite[pages 97, 98]{Jc2}, modulo minor obvious missprints. 

\vspace*{0,15 cm}
\noindent $\bullet$ Intrinsic rolling of $\cD$ on $T_0\cD$.

We are now in conditions to apply the theory developed at the beginning of this section for the intrinsic rolling of $M=\cD$ on $\widehat M=T_0\cD$. Let $\alpha(t)$ be a curve in $\cD$ satisfying $\alpha(0)=0$. Define $u(t):=\displaystyle \frac{\dot \alpha(t)\, e^{-i\theta (t)}}{1-|\alpha(t)|^2}$, so that the horizontal lift of $\alpha$ to $G$  satisfies $g^{-1}\dot g = \bpm 0&u(t)\\
\overline u(t)&0 \epm =:U(t)$, $g(0)=I$. Notice that $\alpha (t)=\tau_{g(t)}(0)=\frac{b(t)}{\bar a(t)}$.
According to (\ref{eq:CauchyPr}) the curve $\widehat\alpha (t)$ is the solution of the initial value problem
$\dot{\widehat\alpha} (t)=d_e\pi (U (t))=u(t), \,\, \widehat\alpha (0)=0$. 

The isometry $A\colon T_{\alpha(t)}M\to T_{\hat \alpha(t)}\widehat M$ is obtained explicitly using~\eqref{eq:isom} and it is given by. 
$$
A(t) v(t)=\big(d_o\tau_{g(t)}\big)^{-1}=\overline a(t)^2v(t),\quad v(t)\in T_{\alpha(t)}M.
$$
\noindent $\bullet$   Extrinsic rolling of $\cD$ on $T^{\text{aff}}_0\cD$

The rollings of $\cD$ can be also represented "extrinsically" after  embedding $\cD$  in a vector space  and defining an appropriate representation of $G$. 
Here, we consider the embedding of $\cD$ in $V=\R^{1,2}$, which is $\R^3$ equipped  with the Minkowski metric $dm^2=-(dx_1)^2+(dx_2)^2+(dx_3)^2$. $V$ is isometric to $(su(1,1), \left< .,. \right>)$. 
The isometric diffeomorphism   
\be\label{Mink}
\begin{array}{lccl}
\iota:&\cD &\rightarrow &\R^{1,2}\\
&z=x_2+ix_3&\mapsto &\iota(z)=(\frac{1+|z|^2}{1-|z|^2},\frac{2x_3}{1-|z|^2},\frac{-2x_2}{1-|z|^2})
\end{array}\ee 
is obtained via the hyperbolic stereographic projection through the point $(-1,0,0)$ and an appropriate change of coordinates.
Then
\be \iota(\cD)=\cH^2=\{(x_1,x_2,x_3):\, x_1^2=1+x_2^2+x_3^2, \, x_1>0\}.\ee

Now  define $\overline G=\Ad_G$. It is known that~\cite{H,ONeill} that $\overline G \subset \SO(V)=\SO^+(1,2)$, that is the connected identity component of
$$\SO(1,2)=\{X\in \SL(3, \R : X^T I_{1,2}X=I_{1,2}\},\,\,   \,\,I_{1,2}= \bpm -1&0&0\\
0&1&0\\
0&0&1\epm.$$
Indeed, calculating $gA_jg^{-1}$, $j=1,2,3$, with $g=\bpm a&b\\\bar b&\bar a\epm, |a|^2-|b|^2=1$, we obtain
\be
Ad_g=\left[\begin{array}{c|cc} |a|^2+|b|^2&2Im(\bar a b)&-2Re(a\bar b)\\\hline
2Im(ab)&Re(a^2-b^2)&-Im(a^2+b^2)\\
-2Re(ab)&Im(a^2-b^2)&Re(a^2+b^2)\end{array} \right]. 
\ee
It can be shown that  
$\Ad_g\, I_{1,2}\, \Ad_g=I_{1,2}$, for all $g\in G$, and the determinant of the diagonal blocks is in both cases equal to $|a|^2+|b|^2>0$, so $\Ad_g\in \SO^+(1,2)$.
It follows that $\dot g(0)\mapsto ad_{\dot g(0)}$ defines  is a Lie algebra isomorphism $d_e\rho$, between $\fg =su(1,1)$ and $\bar \fg =\fso(1,2)$. Since $[A_1,A_2]=A_3$, $[A_1,A_3]=-A_2$, and $[A_2,A_3]=-A_1$, an easy calculation  yields 
 \be
 A=\frac{1}{2}\bpm iv&u\\\bar u&-iv\epm \mapsto ad_A=\bpm 0&u_2&-u_1\\u_2&0&-v\\-u_1&v&0\epm, u=u_1+i u_2.\ee
We also have the Cartan decomposition $\fso(1,2)= \overline \fh \oplus \overline \fp$, 
 where 
 $$
\overline \fh =\spn \left\{ \bpm 0&0&0\\0&0&1\\0&-1&0\epm \right\}, \quad \overline \fp =\spn \left\{ \bpm 0&0&1\\0&0&0\\1&0&0\epm ,  \bpm 0&1&0\\1&0&0\\0&0&0\epm \right\},$$
$\overline\fh$ is the Lie algebra of the isotropy subgroup of $\SO(1,2)$ at $e_1$.

We also need to guarantee that the embedding $\iota$ is equivariant rela\-tive to $\overline G$, i.e., 
$\iota(\tau_g(z))=\Ad_g(\iota(z))$, for every $z\in\cD$ and  $g\in G$. We first show that this identity is true for $z=0$, and then use the transitive action of $G$ on $\cD$ to prove the general case. 
\be \label{equiv}
\begin{array}{ll}
 \iota(\tau_g(0))&=\iota (\frac{b}{\bar a})= 
\iota (\frac{1}{|a|^2}Re(ab),\frac{1}{|a|^2}Im(ab))
\\
&
\\
&
=\left ( |a|^2+|b|^2,2Im(ab), -2Re(ab)\right )\\
&\\
&=\Ad_g(e_1)=\Ad_g(\iota(0)).
\end{array}
\ee
Now, let $h:=\frac{1}{\sqrt{1-|z|^2}}\bpm 1&z\\\bar z&1\epm \in \SU(1,1)$, so that $z=\tau_h(0)$. Using this and the identity (\ref{equiv}), we can write, for each $g\in G$ and  $z\in \cD$,
\be \label{equiv1}
\begin{array}{ll}
 \iota(\tau_g(z))&=\iota(\tau_g(\tau_h(0)))= \iota(\tau_{gh}(0))= \Ad_{gh}(\iota (0))\\
&\\
&=\Ad_g(\Ad_h(\iota (0)))=\Ad_g(\iota(\tau_h(0))=\Ad_g(\iota(z)).

\end{array}
\ee

We are finally in conditions to deal with the extrinsic rolling of the hyperboloid $\cH$ on its affine tangent space at $e_1$, resulting from the action of $\SO^+(1,2)$. Since $\cH^2$ is co-dimension $1$, Corollary~\ref{cor-codim} applies 
and $(\overline R(t), \overline s(t))$ is a rolling map along the curve $\overline \alpha (t)=\overline g (t) e_1$. The kinematic equations for the extrinsic rolling  of $\cH^2$ on $T^{\text{aff}}_{e_1}\cH$ are,
\be  
\begin{array}{ll}
\left\{\begin{array}{l}
\dot{\overline R} (t)=-\overline U(t)\overline R (t)\\
\\
\dot{\overline s} (t)=\overline U(t)e_1
\end{array}\right. , &\quad \overline U=\bpm 0&u_1&u_2\\u_1&0&0\\u_2&0&0\epm
\end{array}.
\ee
 This agrees with the results reported in~\cite{JZ}.

\subsubsection {The projective complex plane and the Riemann sphere.}

 This is another example where the natural geometry on $M$ is induced by the structure of $G$.  The rollings of the projective space $\IC\cP^1$, identified with the extended complex plane $\comp  \cup \infty$,  on its tangent planes can be obtained  essentially in the same way as in the case of the Poincar\'e disk, with obvious adaptations. For this reason we omit certain details here. 
 
Consider the projective plane $M=\IC\cP^1$ with the elliptic metric given in coordinates $(x_1,x_2)$ by $l^2=4\frac{(dx_1)^2+(dx_2)^2}{\left( 1+(x_1^2+x_2^2)\right) ^2}$. The Lie group $G=\SU(2)$ acts transitively on $M$. The isotropy group of the origin $z=0$ is $H=\{\bpm a&0\\0&\bar a\epm,|a|=1\}$. We endow $\fg=\fsu(2)$ with  the metric   $\langle X,Y\rangle =-2\tr(XY)$.
 Relative to this metric, the matrices
\be\label{pauli2}
A_1=\frac{1}{2}\bpm i&0\\0&-i\epm, A_2=\frac{1}{2}\bpm 0&1\\-1&0\epm, A_3=\frac{1}{2}\bpm 0&i\\i&0\epm 
\ee form an orthonormal basis of $\fg$. The Lie algebra $\fh$ and the complementary space $\fp$ are given by
$$\fh=\{\frac{1}{2}\bpm iv&0\\0&-iv\epm:v\in \mathbb R\}, \quad \fp =\{\frac{1}{2}\bpm 0&u\\-\bar u&0\epm, u\in\IC\}. $$
The horizontal lift of a curve $\alpha(t)=x_1+ix_2$ in $\IC\cP^1$ to  $\SU(2)$ is given by $g(t)=\frac{1}{\sqrt{1+|\alpha(t)|^2}}\bpm 1&\alpha(t)\\-\bar \alpha(t)&1\epm e^{\theta (t) A_1}$, with $\theta $ being a solution of  $\dot\theta=\frac{2}{1+|\alpha|^2}( x_1 \dot x_2-\dot x_1x_2)$.

\noindent $\bullet$ Intrinsic rolling of $\IC\cP^1$ on $T_0\IC\cP^1$

We are ready  to deal with the intrinsic rolling of $M=\IC\cP^1$ on its tangent space  at $z=0$.  Let $\alpha(t)$ be a curve in $\IC\cP^1$ satisfying $\alpha(0)=0$, and define $u(t):=\displaystyle \frac{\dot \alpha(t)\, e^{-i\theta (t)}}{1+|\alpha(t)|^2}$, so that the horizontal lift $g(t)\in \SU(2)$ of $\alpha$ satisfies $g^{-1}\dot g = \bpm 0&u(t)\\
-\overline u(t)&0 \epm =:U(t) \in \fp$, $g(0)=I$. 
So, according to Proposition \ref{prop:rolling},   the curve $\alpha(t)$ in $\IC\cP^1$ rolls on the curve $\widehat\alpha (t)$ in  $T_0\IC\cP^1$ which is the solution of $\dot{\widehat\alpha} (t)=u(t)$, $\widehat\alpha (0)=0$.

The isometry (that preserves the elliptic metric)
 is given explicitly by $$
\begin{array}{lccc}
A=(d_0\tau_{g(t)})^{-1}: &T_{\alpha (t)}M &\rightarrow &T_{\widehat\alpha (t)}\widehat M\\
&v(t)&\mapsto & v(t)\overline a(t)^2
\end{array}.$$
So,  $(\alpha,\widehat\alpha, A)$ is a rolling curve for the intrinsic rolling of $\IC\cP^1$ on its tangent space at $0$.


\noindent $\bullet$ Extrinsic rolling of $\IC\cP^1$ on $T^{\text{aff}}_0\IC\cP^1$

 For the extrinsic rolling, we embed  $\IC\cP^1$ in the $3$-dimensional Euclidean space, through the passage to the Riemann sphere $S^2$ via the inverse of the stereographic projection and a change of coordinates. This isometric embedding is defined by
 \be\label{Stereo}
\begin{array}{lccl}
\iota :&\IC\cP^1 &\rightarrow &\R^3\\
&z=x+iy&\mapsto &\iota(z)=\displaystyle (\frac{-2x}{1+|z|^2},\frac{|z|^2-1}{1+|z|^2},\frac{-2y}{1+|z|^2})
 \end{array}.\ee 
Clearly $\iota (\IC\cP^1 )=S^2$, and  $\infty$ is mapped to the north pole of $S^2$.

\noindent In this case 
\be
\Ad_g=\left[\begin{array}{ccc} |a|^2-|b|^2&-2Im(\bar a b)&2Re(\bar a b)\\
2Im(ab)&Re(a^2+b^2)&-Im(a^2-b^2)\\
-2Re(ab)&Im(a^2+b^2)&Re(a^2-b^2)\end{array} \right] \in \SO(3), 
\ee
so, we define $\rho(\SU(2))=\overline G=\Ad_G =\SO(3).$
The Lie algebra isomorphism $d_e\rho\colon \fsu(2)\to \fso(3,\R)$ is defined by
 \be
 A=\frac{1}{2}\bpm iu_1&u_2+iu_3\\-u_2+iu_3&-iu_1\epm \mapsto ad_A=\bpm 0&-u_3&u_2\\u_3&0&-u_1\\-u_2&u_1&0\epm.\ee
Clearly, $\overline\fp =span \left\{ \bpm 0&0&1\\0&0&0\\-1&0&0\epm , \bpm 0&0&0\\0&0&1\\0&-1&0\epm \right\}$.
 Since the embedding  defined in \eqref{Stereo}
is equivariant relative to the adjoint group $\SO(3)$,
 we can finally apply Corollary \ref{cor-codim} to obtain  the  extrinsic rolling of the Riemann sphere on its affine tangent space at the south pole $-e_3$, along the curve $\overline \alpha (t)=\overline g (t) (-e_3)$, where  $\overline g (t)$ is horizontal. Assume that 
$$\overline g (t)^{-1}\dot{\overline g }(t)=\overline U (t)=\bpm 0&0&u_1(t)\\0&0&u_2(t)\\-u_1(t)&-u_2(t)&0\epm
$$
Then, the kinematic  equations are:
\be  
\left\{\begin{array}{lcl}
\dot{\overline R} (t)&=&-\overline U(t)\overline R (t)\\
&&\\
\dot{\overline s} (t)&=&-\overline U(t)\, e_3
\end{array}\right. , 
\ee
 with $\overline U $ as above. These equations are the same  as the equations for the ball-plate problem \cite{J}, or  the equations for the sphere rolling on a plane~\cite {HL,Jc1}.

\subsubsection{Rolling semi-Riemannian orthogonal groups}

Here we consider $M$ to be the connected component containing the identity of the semi-Riemannian orthogonal group $O(p,n-p)$, $ 1\leq p\leq n-1$, consis\-ting of invertible $n\times n$ real matrices $P$, satisfying $P^JP=I_n$, where $J=\diag(I_{p}, -I_{n-p})$, and $P^J:=J^TP^TJ$. The Lie algebra of $O(p,n-p)$, denoted by $\so(p,n-p)$, consists of  $n\times n$ matrices $B$ satisfying $B^J=-B$. If we consider $P\in O(p,n-p)$ partitioned as 
$
 P= \left[
\begin{array}{c|c}
                     P_1&  P_2 \\\hline
                   P_3  & P_4
                   \end{array}
\right],  \mbox{ where $ P_1$ is $p\times p$},  
$ then
 {\small  \begin{equation}\label{SRM}
  \ba{lcl}
  M&=&\SO^+(p, n-p)\\
  &=&\left\{ P\in O(p,n-p):  \det(P)=1, \det(P_1)>0, \det(P_2)>0 \right\} .
  \ea
 \end{equation}
} 
We consider $M$ 
equipped with the semi-Riemannian metric defined by 
\begin{equation}\label{SR-metric}\left< B,C\right>  _J:=\tr{(B^JC)}.
\end{equation}

Consider the Lie group $G:=\SO^+(p, n-p)\times \SO^+(p, n-p)$, equipped with the natural semi-Riemannian metric induced by  (\ref{SR-metric}) on each component, which is bi-invariant. $G$ acts transitively on $M$ with action
\begin{equation}\label{SR-action}
\ba{cccc}\tau : &G\times M&\rightarrow &M\\
&((Q_1,Q_2),P)&\mapsto &Q_1PQ_2^{-1}
\ea.
\end{equation}

Fixing a point $P_0\in M$, the projection $\pi:G \rightarrow M$ maps $(Q_1,Q_2)$ to $Q_1P_0Q_2^{-1}$.
The isotropy subgroup at $P_0$ is
\begin{equation}\label{SR-isotropy}
H=\left\{ (Q_1, Q_2)\in G:  Q_1P_0Q_2^{-1}=P_0 \right\},  
\end{equation}
and $M=G/H$. Of course, the semi-Riemannian metric (\ref{SR-metric}) on $M$ is also $Ad_H$-invariant. The Lie algebra $\fg= \so(p,n-p)\oplus \so(p,n-p)$ splits as
$\fg= \fh \oplus \fp$, where
\be\label{SR-CD}
\ba{l}
\fh = \left\{ (B, P_0^{-1}BP_0):  B\in \so(p,n-p) \right\} \\
\\
\fp =\left\{ (C, -P_0^{-1}CP_0):  C\in \so(p,n-p) \right\} 
\ea , 
\ee
and this orthogonal splitting satisfies (\ref{eq:Cartan}).

 \noindent $\bullet$  Intrinsic rolling of $M=\SO^+(p, n-p)$ on $\M=T_{P_0}\SO^+(p, n-p)$.
 
We now apply the results obtained in Section \ref{RolSymSp} for the intrinsic rolling of $M=\SO^+(p, n-p)$ on its tangent space at the point $P_0$. Note that  the differential of $\pi$ at $(e,e)$, the identity in $G$,
is given by
\begin{equation}\label{1000}
\ba{cccc}d_{(e,e)}\pi : &\fg&\rightarrow &T_{P_0}M\\
&(U_1,U_2)&\mapsto &U_1P_0-P_0U_2
\ea, 
\end{equation}
and the kernel of $d_{(e,e)}\pi$ is $\fh$. So, $d_{(e,e)}\pi$ defines an isomorphism between $\fp$ and  $T_{P_0}M$, mapping $(U,-P_0^{-1}UP_0)$ to $2UP_0$.

Let $\alpha (t)$ be a curve in $M$ satisfying $\alpha(0)=P_0$, and $Q(t)$ a horizontal lift of $\alpha (t)$ to $G$, i.e., $\pi (Q(t))=\alpha (t)$ and $Q^{-1}\dot{Q}= (U(t), -P_0^{-1}U(t)P_0)$, for some curve $U(t)\in \so(p,n-p)$. Then, according to (\ref{eq:CauchyPr}), Section \ref{RolSymSp}, the curve $\alpha(t)\in M$ rolls on the curve $\widehat\alpha(t)\in \M$ defined by
$$
\dot{\widehat\alpha}(t)=2U(t)P_0, \qquad \widehat\alpha(0)=0, 
$$
and the isometry $A(t)$ is defined in  (\ref{eq:isom}) as the inverse of $d_{P_0}\tau_{Q(t)}$. Since for $Q=(Q_1,Q_2)$,
\begin{equation*}
\ba{cccl}d_{P_0}\tau_{Q} : &T_{P_0}M&\rightarrow &T_{Q_1P_0Q_2^{-1}}M\\
&CP_0&\mapsto &Q_1CP_0Q_2^{-1}= Q_1CQ_1^{-1}Q_1P_0Q_2^{-1}
\ea, 
\end{equation*} 
where $C\in \so(p,n-p)$, with the identification of the vector spaces $T_{P_0}M$ and $T_{\widehat\alpha (t)}M(T_{P_0}M)$, we finally obtain
\begin{equation*}
\ba{cccl}A(t) : &T_{\alpha (t)}M&\rightarrow &T_{\widehat\alpha (t)}\M\\
&DQ_1P_0Q_2^{-1}&\mapsto &Q_2^{-1}DQ_2P_0, \quad D\in \so(p,n-p)
\ea. 
\end{equation*}
In conclusion, the triple $(\alpha(t), \widehat\alpha (t), A(t))$ is a rolling curve in the sense of Definition \ref{def_intrinsic}.\\

  \noindent $\bullet$ Extrinsic rolling of $M=\SO^+(p, n-p)$ on $\M=T^{\text{aff}}_{P_0}\SO^+(p, n-p)$.
We isometrically embed $M$ and $\M$ on the semi-Euclidean vector space 
$ 
V=\left( \fgl(p,n-p), \left< .\,,.\right>  _J \right)$, and identify $\iota (M)$ and $\iota (\M)$ with $M$ and $\M$ respectively. In this case, also the representation $\rho$ of $G$ on $V$ is the identity map, so we can write everything in Proposition \ref{prop:rollingExt} without using overlines. The equivariance property (\ref{eq:ext5a}) is also trivially satisfied,  and the action of $G$ on $V$ is simply the extension of the action (\ref{SR-action}) from $M$ to the embedding space $V$. Notice that 
\be \label{TN-SM}\ba{c}
T_{P_0}\SO^+(p, n-p)=\{ BP_0,\,\, B^J=-B\} , \\
\\
T^{\perp}_{P_0}\SO^+(p, n-p)=\{ CP_0,\,\, C^J=C\}.
\ea
\ee

 We now find the rolling map $g(t)=(R(t),s(t))\in SE(V)$ along the curve $\alpha(t)=Q_1(t)P_0Q_2^{-1}(t)$ in $M$, satisfying $\alpha(0)=P_0$ where $Q(t)=(Q_1(t),Q_2(t))$ is a horizontal lift of $\alpha(t)$ to $G$. So,
 $Q^{-1}\dot{Q}\in \fp$, that is, $Q^{-1}\dot{Q}=(Q_1^{-1}\dot{Q_1}, -P_O^{-1} Q_1^{-1}\dot{Q_1} P_0)$. Defining 
 $U:=Q_1^{-1}\dot{Q_1}$, we have $Q^{-1}\dot{Q}=(U, -P_O^{-1} U P_0)$, and after a few simple calculations, we get
$\dot\alpha (t)=Q_1(t)(2U(t)P_0)Q_2^{-1}(t)$. So, according to Proposition \ref{prop:rollingExt}, $s(t)$ is the only solution of 
$$
\dot{s}(t)=2U(t)P_0, \qquad s(0)=0.
$$
We also know that, for every tangent vector field $X(t)$ along $\alpha (t)$ \begin{equation*}
d_{\alpha(t)} R (t)\vert_{T_{\alpha(t)} M}(X (t))=Q(t)^{-1}\,  X (t),
\end{equation*}
and the tangent no-twist condition is satisfied. 
For a general symmetric space this is not enough to define a rolling map in the sense of Definition  \ref{def_extrinsic}, because the normal no-twist condition also requires that we know how to define 
$d_{\alpha(t)} R (t)\vert_{T^{\perp}_{\alpha(t)} M}$. However, for this particular example, it turns out that if $d_{\alpha(t)} R (t)=Q(t)^{-1}$,  the normal no-twist condition is also satisfied.

To show this we rewrite the normal no-twist condition  5 of Definition \ref{def_intrinsic} in its equivalent form given in 5' of Proposition \ref{prop-sharpe}.
Taking into consideration that $$T_{\widehat\alpha (t)}\M=T_{P_0}\SO^+(p, n-p) \,\,  \mbox{ and } \, \, T^{\perp}_{\widehat\alpha (t)}\M=T^{\perp}_{P_0}\SO^+(p, n-p),$$ 
and using (\ref{TN-SM}), the normal no-twist condition is equivalent to prove that for every $B=-B^J$, 
$(R^{-1}\dot{R})(BP_0)$ is always of the form $CP_0$, for some matrix $C$ satisfying $C=C^J$. 
 But $R^{-1}\dot{R}=(U, -P_O^{-1} U P_0)$ and  $U^J=-U$, so 
 $$(R^{-1}\dot{R}).(BP_0)=UBP_0+BP_0P_O^{-1} U P_0= (UB+BU)P_0=CP_0$$ where $C=UB+BU=(UB+BU)^J=C^J$. 
 Writing $R=(R_1,R_2)$, the kinematic equations are:
 \be
 \left\{
 \ba{lcl}
 \dot{s}(t)&=&2U(t)P_0\\
 \dot{R_1}(t)&=&-U(t)R_1 \\
 \dot{R_2}(t)&=&P_O^{-1} U(t) P_0R_2(t)
  \ea
 \right.,
 \ee
with initial conditions $s(0)=0,\, R(0)=(e,e)$. This coincides with the results in \cite{CL2012}.


   \section{Rolling of Stiefel manifolds on the affine tangent space} \label{RollingStiefel}
   
 
We will now narrow our discussion to  the Stiefel manifolds $\St_{nk}$ equipped with the  Riemannian metric inherited from the ambient Euclidean vector space $\mathcal M_{nk}$ consisting of $n\times k$  matrices. Rolling motions of Stiefel manifolds were already studied in \cite{HKL}, but here we present an alternative approach which is coordinates free, and also a different representation of Stiefel as a homogeneous space is used.
     
 There are two compelling  reasons for including the Stiefels in  this paper.  Firstly, because   it is the only case  outside of the $G$-invariant Riemannian manifolds where the rolling equations are explicitly calculated, and  secondly because it  illustrates the relevance  of  the normal no-twist condition for rolling manifolds that are homogeneous spaces but not symmetric spaces.


\subsection{Stiefel manifold}


Let $\R^n$ be the Euclidean vector space with its standard scalar product. We denote by $\fgl(n)$ the vector space of all real $n\times n$ matrices endowed with the positive definite scalar product $\langle N,M\rangle=\tr(N^TM)$. We induce this metric on the subspace $V=\mathcal M_{nk}\subset \fgl(n)$ of $n\times k$  matrices. We also consider the groups $\GL(n)$ and $\SO(n)$ as submanifolds of $\fgl(n)$. We define the action of $\fgl(n)$ on $V$ through the linear isometric homomorphism 
$$
\begin{array}{ccccll}
\rho&\colon &\fgl(n)&\to&\fgl(V)
\\ 
&&A&\mapsto &
\rho_A\end{array}
$$
with $\rho_A(M)=AM$, $M\in V$. Under this convention, we obtain $\frac{d}{dt}\rho_{A(t)}=\rho_{\dot{A}(t)}$ for any smooth enought curve $A(t)$ in $\fgl(n)$. Moreover, 
the restriction of $\rho$ on the group $\SO(n)$ is a group homomorphism $\rho\colon \SO(n)\to\SO(V)$ meaning that $\rho_{Q^{-1}}=(\rho_Q)^{-1}$ and $\rho_{Q_1Q_2}=\rho_{Q_1}\circ\rho_{Q_2}$.

The Stiefel manifold $\St_{nk}$ consists of ordered sets of $k$-orthonormal vectors in  $\R^n$. Any ordered set $m_1,\dots,m_k$ of orthonormal vectors can be identified with a matrix $M$  whose  columns  are $m_1,\dots,m_k$. Any such matrix  $M$ satisfies  $M^TM=I_k$, where $I_k$ is the $k$-dimensional identity matrix, and $M^T$ is the matrix transpose of $M$. This matrix representation realizes $\St_{nk}$ as an isometrically embedded compact submanifold of the  Euclidean vector space $V=\mathcal M_{nk}$, which we continue to denote by $\St_{nk}$. 

The Stiefel manifold can be also viewed as a homogeneous space. In what follows   $\{e_1,\dots,e_n\}$  denotes the standard basis in $\R^n$ and  $E$  denotes the matrix with  columns $e_1,\dots,e_k$. The group $\SO(n)$ acts transitively on $\St_{nk}$. Thus $\St_{nk}$ can be identified with the orbit $\{\rho_Q(E):Q\in \SO(n)\}$. The   isotropy subgroup $H=\{Q\in \SO(n): \rho_Q(E)=E\} $  reduces to matrices  $Q=\bpm I_k&0\\0&X\epm$, with $X\in \SO(n-k)$.  Evidently $H$ is isomorphic to $\SO(n-k)$ and consequently $\St_{nk}=\SO(n)/\SO(n-k)$.
It follows that  $\fso(n)=\fp\oplus\fh$, where 
\be \label{structure}
 \begin{array}{cc}\fh=\left\{\bpm 0&0\\0&C\epm, \, C\in \fso(n-k)\right\}, \\&\\ \fp=\left\{\bpm A&-B^T\\B&0\epm, \, A\in \fso(k),\, B\in \mathcal M_{(n-k)k}\right\}.\end{array}
 \ee 
 One can easily verify that $\fh$ is the Lie algebra of $H$, $\fp$  is    the   orthogonal complement to $\fh$ relative to the trace metric $\langle M_1,M_2\rangle=-\tr(M_1M_2)$, and  
 \be 
 [\fh,\fh]\subset \fh,\quad [\fp,\fh]=\fp,\quad \fh\subset[\fp,\fp].\ee 
The latter algebraic properties show that the Stiefel manifold $\St_{nk}$ is not a symmetric space. It does not allow to use the general aproach for the construction of the rolling, which we developed for symmetric spaces in Section~\ref{RolSymSp}. 

The projection map $\pi\colon \SO(n)\to \St_{nk}=\SO(n)/\SO(n-k)$ is given by $\pi(Q)=QE=\rho_Q(E)$. It is a submersion, and $d_{e}\pi\colon \fp \rightarrow T_ESt_{nk}$ is an isomorphism, mapping $U\in \fp$ to $UE$.  
The group homomorphism $\rho\colon \SO(n)\to\SO(V)$ induces the Lie algebra homomorphism $\fso(n)\to\fsl(V)$. We will use the same notation $\fp$, $\fh$ for the images of $\fp$, $\fh$ under the Lie algebra homomorphism.
 
 It follows, see for instance \cite{edel98}, that the tangent to $\St_{nk}$ at a point $P\in \St_{nk}$ is given by
   \begin{equation} \label{St_tangentspace}
   T_P\St_{nk}=\{W\in \mathcal M_{nk}: \, W^TP+P^TW=0 \}.
   \end{equation}
   \begin{remark}\label{97}
   A simple calculation using~\eqref{St_tangentspace} shows that for  $Q\in G$, we have $T_{QP}\St_{nk}=QT_P\St_{nk}$. 
   \end{remark}
 In particular, 
\be \label{tgspace}T_E\St_{nk}=\left\{ \bpm A\\B\epm , \, A\in \fso(k),B\in \mathcal M_{(n-k)k} \right\} \subset \mathcal M_{nk}.\ee

Hence, its orthogonal complement in $\cM_{nk}$ is \be \label{98}T_E^\perp \St_{nk}=\left\{\bpm S\\0\epm, \, S\in \mathcal M_{kk}, \,S^T=S\right\}.\ee

 The orthogonal complement is further decomposed as
 \be \label{orth_decomp}
T_E^\perp \St_{nk}=V_E \oplus sl_{kk},
\ee
where $V_E$ is the linear span of $E$, and its orthogonal complement, denoted by $sl_{kk}$ is  defined as
$$
sl_{kk}=\left\{ \bpm S\\0\epm, \, S\in \mathcal M_{kk}, \,S^T=S, \, \tr(S)=0 \right\} .
$$
  
\subsection{Extrinsic rolling of $\St_{nk}$} 

  We will now turn our attention to  the rollings of curves in $M=\St_{nk} $ on $\M=T^{\text{aff}}_E\St_{nk}:=E+T_E\St_{nk}$ which is the affine tangent space at $E$. According to  Definition~\ref{def_extrinsic},   curves $\alpha(t)$ in $\St_{nk}$ are rolled on curves $\widehat\alpha(t) $ in $\widehat M$  by   rolling maps   $g(t)=(R(t),s(t))$ in $\SE(V)=\SO(V)\ltimes V$  under the action  $g(t)(\alpha (t))=R(t)(\alpha(t))+s(t)=\widehat\alpha(t)$. The fact that $\SO(n)$ acts transitively on $\St_{nk}$ implies that  there is a unique horizontal  curve $ Q(t) \in 
  \SO(n)$, $Q(0)=I$, that projects on $\alpha(t)$, that is, $\rho_{Q(t)}(E):=Q(t)E=\alpha (t)$, and $Q^{-1}(t)\dot{Q}(t)\in \fp$.
  
We will now assume that  $R(t)^{-1}=\rho_{Q(t)}\circ S(t)$  for some curve $S(t)$ in the isotropy group $K=\{S\in \SO(V):\ S(E)=E\}$. The choice of $Q(t)$ as a horizontal lift of $\alpha(t)$ allows particularly satisfy the first rolling condition in Definition~\ref{def_extrinsic}. Namelly: $R(t)^{-1}(E)=\rho_{Q(t)}\circ S(t)(E)=\rho_{Q(t)}(E)=\alpha (t)$, and therefore $R(t)(\alpha (t))=E.$
But then
 $ g(t)(\alpha (t))=R(t)(\alpha (t))+s(t)=E+s(t)=\widehat\alpha (t)$ by the reguirement of the first rolling condition in Definition~\ref{def_extrinsic}. It implies  that 
 \be s(t)\in T_E\St_{nk}, \,\, \text{ and so }\,\, \dot s(t)=\dot{\widehat\alpha} (t).\ee
 
In what follow we will find the condition on $S(t)$ such that $g(t)=\big(R(t),s(t)\big)$ satisfies the no-slip and both no-twist constrains. According  to the second rolling condition  in Definition~\ref{def_extrinsic}, \be d_{\alpha (t)}g(t)(T_{\alpha(t)}\St_{nk})=R(t)(T_{\alpha(t)}\St_{nk})=T_E\St_{nk}.\ee 
Remark~\ref{97} and $\rho_{Q(t)}(E)=\alpha (t)$ lead to 
\bean d_E\rho_{Q(t)}(T_E\St_{n,k})=\rho_{Q(t)}(T_E\St_{n,k})=T_{\alpha (t)}\St_{n,k}.\ean Therefore,
  \begin{eqnarray*}&
 T_E\St_{nk} =R(t)(T_{\alpha(t)}\St_{nk})=S^{-1}\circ \rho_{Q^{-1}}(T_{\alpha(t)}\St_{nk})=S^{-1}(T_E\St_{n,k}).\end{eqnarray*}
Hence, $S(T_E\St_{nk})=T_E\St_{nk}$, and since $S$ is an isometry in $V$, we also have $S(T_E^\perp \St_{nk})=T_E^\perp \St_{nk}$. So, 
\be \label{99}
S(T_E\St_{nk})=T_E\St_{nk}, \quad S(T_E^\perp \St_{nk})=T_E^\perp \St_{nk}.
\ee
 Moreover, since $S(E)=E$ and $S$ is an  orthogonal transformation,
 \be \label{100}
 S(V_E)=V_E, \qquad S(sl_{kk})=sl_{kk}.
 \ee

 From now on,  we use the notation $\dot R(t)$ for the time derivative of $R(t)\in \SO(V)$, and similarly for the time derivative of any other curves in  $\SO(V)$. We also recall from the beginning of this section that $\dot{\rho_{Q(t)}}:= \rho_{\dot{Q}(t)}$ for $Q(t)\in \SO(n)$.
 
The no-slip condition requires that  $\dot{R}(t)(\alpha (t))+\dot s(t)=0$, or,
\be \dot{R}(t)\circ R(t)^{-1}(E)=-\dot s(t).  \ee
Since $R=S^{-1}\circ \rho_{Q^{-1}}$, we have
\be
\ba{lcl}\dot{R}\circ R^{-1}&=&(\dot{S^{-1}} \circ \rho_{Q^{-1}}+{S^{-1}}\circ \rho_{\dot{Q^{-1}}})\circ \rho_{Q}\circ S\\& =&\dot{S^{-1}}\circ S+S^{-1}\circ \rho_{\dot{Q^{-1}}}\circ \rho_{Q}\circ S \ea
\ee
  Since   $\dot{S^{-1}}\circ S= -S^{-1}\circ \dot{S}$ and   $\rho_{\dot{Q^{-1}}} \circ \rho_Q=- \rho_{Q^{-1}} \circ \rho_{\dot{Q}}=- \rho_{Q^{-1}\dot{Q}}$,  the above can be rewritten as
 \be\label{Rolcurve}
\dot{R}\circ R^{-1}=-S^{-1}\circ \dot{S}- S^{-1}\circ (\rho_{Q^{-1}\dot{Q}})\circ S.
\ee
 Note that $S(E)=E$ implies $\dot{S}(E)=0$,  and  $\rho_{Q^{-1}\dot{Q}}(E) =Q^{-1}\dot Q E =UE$, for $U\in \fp$. Taking into consideration  the structure of elements in $\fp$, appearing in (\ref{structure}), $Q^{-1}(t)\dot Q(t)=\bpm A(t)&-B^T(t)\\B(t)&0\epm$, and consequently the no-slip condition requires that 
\be\label{dot-s}
\dot s(t)=-\dot R(t)R^{-1}(t)(E)=S^{-1}\Big(\bpm A(t)\\B(t)\epm\Big). \ee
   
   We will now choose  $S(t)\in K$, or equivalently $\Omega(t)=\dot S \circ S^{-1}\in \fso(V)$ so that $R(t)$ satisfies the no-twist conditions. 
  Since $\Omega(t)(E)=0$.  Therefore 
\be \label{102}
\begin{array}{lll}
&\dot S \circ S^{-1}(T_E \St_{nk})=\Omega(t)(T_E\St_{nk})\subset T_E\St_{nk},
\\
&\dot S\circ S^{-1}(T_  E^\perp \St_{nk})=\Omega(t)(T_ E^\perp \St_{nk})\subset T_ E^\perp \St_{nk}.
\end{array}
\ee
\be \label{101}\Omega \, (V_E)=0, \quad \Omega \, (sl_{kk})
 \subset sl_{kk}. \ee
due to~\eqref{99} and~\eqref{100}.

 Now, since $T_{\widehat\alpha(t)}\widehat M = T_E\St_{nk}$, and similarly $(T_{\widehat\alpha (t)}\widehat M)^{\perp} = T_E^{\perp}\St_{nk}$,  the tangential  no-twist condition 4' given in Proposition \ref{prop-sharpe}, requires that
 \be \label{tnt}
 \dot g(t)\circ g^{-1}(t)(T_E\St_{nk})\subset T_E^\perp \St_{nk}.
 \ee
 Since $\dot g\circ g^{-1}(T_E\St_{nk})=\dot  R\circ R^{-1}(T_E\St_{nk})$, taking into account~\eqref{Rolcurve}, we can write
 $$
 \dot g \circ g^{-1}(T_E\St_{nk})=-S^{-1}\circ (\dot S\circ S^{-1}+ \rho_{Q^{-1}\dot Q})\circ S\, (T_E\St_{nk}),
$$
and using~\eqref{99}, the tangential no-twist condition (\ref{tnt}) can be written as 
$$\dot S\circ S^{-1}\, (T_E\St_{nk}) + \rho_{Q^{-1}\dot Q} \, (T_E\St_{nk})\subset 
T_E^\perp \St_{nk},$$
or, equivalently,
\be\label{tang-sol}
\Omega\, (T_E\St_{nk})=  -\Pi \left( \rho_{Q^{-1}\dot Q }\, (T_E\St_{nk})\right),
\ee
where $\Pi$  denotes the orthogonal projections  of $V$ onto $T_E\St_{nk}$.

We now impose the normal no-twist condition 
\be \label{nnt}
 \dot g(t)\circ g^{-1}(t)(T_E^\perp \St_{nk})\subset T_E \St_{nk},
 \ee 
 and similarly to the previous calculations, we obtain a second restriction on $S$:
 \be\label{tang-sol2}
\Omega\, (T_E^\perp  \St_{nk})=  -\Pi^\perp  \left( \rho_{Q^{-1}\dot Q }\, (T_E^\perp \St_{nk})\right),
\ee
where  $\Pi^\perp $  denotes the orthogonal projections  of $V$ onto $T_E^\perp \St_{nk}$.

To make sure that the previous condition can be fulfilled, we must show that the righthand side of~\eqref{tang-sol2} is according to the action \eqref{101} of $\Omega$ on each subspace of the direct decomposition  of $T_E^\perp  \St_{nk}$ in~\eqref{orth_decomp}. For that, we compute the product of  the matrix $Q^{-1}\dot Q$ by elements in $T_E^\perp \St_{nk}$, using the fact that $Q^{-1}\dot Q\in \fp$ and the structure of the matrices in these subspaces, given in~\eqref{structure} and~\eqref{98}.

Assume that 
$$
Q^{-1}\dot Q=\bpm A&-B^T\\ B&0 \epm , \, A= -A^T, \, \mbox{ and take }\,  \bpm X \\0\epm , \, \mbox{ with }\, X=X^T.
$$
Then, 
\begin{eqnarray*}&
  Q^{-1}\dot Q\bpm X \\0\epm=\frac{1}{2}\bpm AX+XA\\2BX\epm+\frac{1}{2}\bpm AX-XA\\0\epm.
  \end{eqnarray*}
 Notice that $AX-XA$ is symmetric with trace zero, and when $X=I_k$, $AX-XA=0$.  Therefore, as required,
 $$\Pi^\perp  \left( \rho_{Q^{-1}\dot Q }\, (V_E)\right)=0, \qquad \Pi^\perp  \left( \rho_{Q^{-1}\dot Q }\, (sl_{kk})\right)\subset sl_{kk}.$$
 
 {\bf We now summarize how to find the rolling map} $(R(t),s(t))\in \SE(V)$, for  rolling $\St_{nk}$ on $T^{\text{aff}}_E\St_{nk}$, along a curve $\alpha(t)$, $\alpha(0)=E$.
 \begin{itemize}
 \item[1.] Find the horizontal lift $Q(t)$ of $\alpha(t)$, satisfying $Q(0)=e_G$. \linebreak
 We know that $
Q^{-1}\dot Q=\bpm A&-B^T\\ B&0 \epm , \, A= -A^T$.

\item[2.] Find $S(t)$ using the no-twist conditions \eqref{tang-sol}, \eqref{tang-sol2}, with  $S(0)=e_{\SO(V)}$. Those conditions can be rewritten as:
$$
\left\{
\ba{lcl}
\dot S \circ S^{-1}(v^\top)=-\Pi ( Q^{-1}\dot Q \, v^\top), \quad \forall v^\top \in T_E \St_{nk}\\
\dot S \circ S^{-1}(v^\perp)=-\Pi^\perp ( Q^{-1}\dot Q \, v^\perp), \quad \forall v^\perp \in T_E^\perp \St_{nk}
\ea
\right. .
$$
\item[3.] Find $R=S^{-1}\circ \rho_{_{Q^{-1}}}$.
\item[4.] Find $s(t)$ by solving equation \eqref{dot-s}, resulting from the no-slip condition, with $s(0)=0$:
$$\dot s(t)=S^{-1}\Big(\bpm A(t)\\B(t)\epm\Big).$$

 \end{itemize}
 
 \section*{Acknowledgments}  The work of second and third authors was partially supported by the project Pure Mathematics in Norway, funded by Trond Mohn Foundation and Troms\o\ Research Foundation.
 The third author also thanks Funda\c{c}\~ao para a Ci\^encia e Tecnologia (FCT) and COMPETE 2020 program for the financial support to the project UIDB/00048/2020.

\end{document}